\documentclass[11pt]{article}
\usepackage[numbers,sort&compress]{natbib}
\usepackage{appendix}
\usepackage{enumerate}
\usepackage{amscd}
\usepackage{amsmath}
\usepackage{latexsym}
\usepackage{amsfonts}
\usepackage{amssymb}
\usepackage{amsthm}
\usepackage{verbatim}
\usepackage{mathrsfs}
\usepackage{enumerate}
\usepackage{hyperref}

\theoremstyle{plain}
\theoremstyle{definition}\newtheorem{theorem}{Theorem}[section]
\theoremstyle{plain}
\theoremstyle{plain}\newtheorem{coro}[theorem]{Corollary}
\theoremstyle{plain}
\theoremstyle{remark}\newtheorem{remark}{Remark}[section]
\usepackage{xcolor}

\newcommand{\Div}{\mathrm{div}\,}
\newcommand{\B}{\Big}

\newcommand{\be}{\begin{equation}}
\newcommand{\ee}{\end{equation}}
 \newcommand{\ba}{\begin{aligned}}
 \newcommand{\ea}{\end{aligned}}

  \newcommand{\f}{\frac}
    
  \newcommand{\ben}{\begin{enumerate}}
   \newcommand{\een}{\end{enumerate}}

\newcommand{\Rmnum}[1]{\expandafter\@slowromancap\romannumeral #1@}

\allowdisplaybreaks

\numberwithin{equation}{section}
\begin{document}
\title{Yaglom's law and conserved quantity dissipation in turbulence}
\author{Yanqing Wang\footnote{   College of Mathematics and   Information Science, Zhengzhou University of Light Industry, Zhengzhou, Henan  450002,  P. R. China Email: wangyanqing20056@gmail.com},  ~  
      \, Wei Wei\footnote{School of Mathematics and Center for Nonlinear Studies, Northwest University, Xi'an, Shaanxi 710127,  P. R. China  Email: ww5998198@126.com }      ~  and   ~\, Yulin Ye\footnote{ School of Mathematics and Statistics,
      	Henan University,
      	Kaifeng, 475004,
      	P. R. China. Email: ylye@vip.henu.edu.cn} }
\date{}
\maketitle
\begin{abstract}
 In this paper, we are concerned with the local exact relationship for  third-order structure functions in  the temperature equation, the inviscid MHD equations  and the Euler equations in the sense of  Duchon-Robert type and Eyink type. It is shown that the local version of Yaglom's $4/3$ law is valid for the dissipation rates of  conserved quantities such as the energy, cross-helicity and helicity in these systems.
    In the spirit of Duchon-Robert's classical  work,  we derive the  dissipation term    resulted from the lack of smoothness of the solutions in corresponding conservation relation.
    It seems that these results suggest that the Yaglom's   law of the hydrodynamic  equations holds if an analogue of dissipation term as Duchon-Robert's is obtained. Base on this, the first  Yaglom's relation for  the Oldroyd-B model and,  inspired by the very recent work due to Boutros-Titi,     six
new 4/3 laws for subgrid scale $\alpha$-models of turbulence are also presented.
  \end{abstract}
\section{Introduction}
\label{intro}
\setcounter{section}{1}\setcounter{equation}{0}
In 1941, Kolmogorov \cite{[Kolmogorov]}   presented the following celebrated exact relation
 \be\label{Kolmogorov45law}
\langle[\delta v_{L}(r)]^{3} \rangle=-\f45\epsilon r,
 \ee
 where $\epsilon$ is the mean rate of
  kinetic energy dissipation per unit mass, $\delta v_{L}(r)=\delta v (r)\cdot\f{r}{|r|}=(v(x+r)-v(x))\cdot\f{r}{|r|}$ stands for the longitudinal velocity increment and $\langle\cdot\rangle$ denotes the mean value. The Kolmogorov law  \eqref{Kolmogorov45law}  plays an important role in the theory of homogenous isotropic turbulence, which can be found in the book by Frisch  \cite{[Frisch]}. And it should be noted that the derivation of Kolmogorov's $4/5$ law in \cite{[Kolmogorov]} rests on K\'arm\'an-Howarth  equation in the statistical sense. Moreover, for  alternative approaches to Kolmogorov's  $4/5$ law, the reader may refer to \cite{[NT],[Eyink1],[Rasmussen]}. Specially,
without assumptions of homogeneity and isotropy, Eyink \cite{[Eyink1]} proved the following version of
Kolmogorov's  $4/5$ law
 \be
 S_{L}(v)=-\f45D(v),
 \ee
where $$S_{L}(v)=-\lim\limits_{\lambda\rightarrow0}S_{L}(v,\lambda)=-
\lim\limits_{\lambda\rightarrow0}\f{1}{\lambda}\int_{\partial B } \ell \cdot\delta v(\lambda\ell) |\delta v_{L}(\lambda\ell)|^{2}\f{d\sigma(\ell) }{4\pi}$$
and
 \be\label{drKHMr}
D(v)=-\lim\limits_{\varepsilon\rightarrow0}\f14
\int_{\mathbb{T}^{3}}\nabla\varphi_{\varepsilon}(\ell)\cdot\delta v(\ell)|\delta v_{L}(\ell)|^{2}d\ell,\ee
here, $\sigma(x)$ stands for the surface measure on the sphere $\partial B=\{x\in \mathbb{R}^{3}: |x|=1\}$ and $\varphi$ is  some smooth non-negative function  supported in $\mathbb{T}^{3}$ with unit integral and $\varphi_{\varepsilon}(x)=\varepsilon^{-3}\varphi(\f{x}{\varepsilon})$. It should be remarked that the  dissipation term $D(v)$  was  first defined by
Duchon-Robert in \cite{[DR]},  in which they studied   the local energy dissipation for the  weak solutions of the Euler equations and the Navier-Stokes equations and
 proved that  the following local energy equation
$$ \partial_{t}(\f12|v|^{2})+\text{div}\B(v(\f{1}{2}|v|^{2}+\Pi)\B)=D(v),
$$
holds in the sense of distributions on $(0,T)\times \mathbb{T}^3$.
Besides, the authors in \cite{[DR]} also showed the following  4/3 law
\be\label{DR4/3law}
S(v)=-\f43D(v),
\ee
where $S(v)=-\lim\limits_{\lambda\rightarrow0}S (v,\lambda)=-
\lim\limits_{\lambda\rightarrow0}\f{1}{\lambda}\int_{\partial B }\ell\cdot\delta v(\lambda\ell)|\delta v(\lambda\ell)|^{2}\f{d\sigma(\ell) }{4\pi}$. It should be noted that  the equation \eqref{DR4/3law} corresponds to the following $4/3$ law
   \be\label{MY4/3law}
\langle\delta v_{L}(r)|\delta v(r)|^{2} \rangle=-\f43\epsilon r,
 \ee
 which is attributed to Monin and Yaglom in \cite{[MY]} and
usually  called as Yaglom's law  in physics.
On the other hand, analogously to Kolomogorov's $4/5$ law, Yaglom in \cite{[Yaglom]} also derived  the following  $4/3$ law when considering  the local structure of the temperature field in a turbulent flow
  \be\label{Yaglom4/3law}
\langle\delta v_{L}(r)|\delta \theta(r)|^{2} \rangle=-\f43\epsilon_{\theta} r,
 \ee
 where $\theta$
 is the temperature field and $\epsilon_{\theta} $ represents the mean dissipation rate of general energy $\f12\theta^{2}$ per unit mass. Furthermore, the $4/3$
law is intimately related with K\'arm\'an-Howarth-Monin relation
\be\label{KHMrelation}
\nabla\cdot\langle\delta v_{L}(r)|\delta v(r)|^{2} \rangle=-4\epsilon,\ee
see \cite{[Frisch],[Podesta],[Monin],[PFS]}. The Yaglom relation for any dimension and Yaglom relation inequality were considered by
Eyink in \cite{[Eyink2]}.
Moreover, in \cite{[Eyink1]}, Eyink  pointed out that \eqref{drKHMr} is closely connected with \eqref{KHMrelation}.

In addition, an analogue of  Yaglom's law \eqref{Yaglom4/3law} in the incompressible magnetohydrodynamic turbulence
   \be\label{MHDYaglom4/3lawElsasser variables}\ba
\langle\delta u_{L}(r)|\delta h(r)|^{2} \rangle=-\f43\epsilon_{+} r,\\
\langle\delta h_{L}(r)|\delta u(r)|^{2} \rangle=-\f43\epsilon_{-} r,\\
 \ea\ee
were discovered by Politano and Pouquent in \cite{[PP1],[PP2]}, where K\'arm\'an-Howarth  type equation  for the MHD equations \eqref{mhdElsasser} in terms of the Els\"asser variables
\be\label{Elsasser}
u=v+b, h=v-b,
\ee
   were obtained and $\epsilon_{\pm}$ stands for the energy
transfer and dissipation rates of the variables $u$ and $h$.  Here, $v$ and $ b$ describe the flow velocity field and the magnetic field, respectively.  Besides, in light of the primitive variables, Politano and Pouquent found that
the exact relationship for  third-order structure functions  relation \eqref{MHDYaglom4/3lawElsasser variables} reduces to
   \be\label{MHDYaglom4/3lawprimitive variables}\ba
   \langle\delta v_{L}(r)|\delta v(r)|^{2} \rangle+\langle\delta v_{L}(r)|\delta b(r)|^{2} \rangle-2\langle\delta b_{L}(r)|\delta v(r)\delta b(r)|  \rangle=-\f43\epsilon_{1} r,\\
   -\langle\delta v_{L}(r)|\delta b_{L}(r)|^{2} \rangle-\langle\delta b_{L}(r)|\delta v(r)|^{2} \rangle+2\langle\delta v_{L}(r)|\delta v(r)\delta b(r)|  \rangle=-\f43\epsilon_{2} r, \\
 \ea\ee
 where $\epsilon_{1} =\f{ \epsilon_{+}+ \epsilon_{-}}{2}$ and
 $\epsilon_{2} =\f{ \epsilon_{+}- \epsilon_{-}}{2}$ denote the  dissipation rates of the total energy  $\int_{\mathbb{T}^3}\f{v^{2}+b^{2}}{2}dx$  and the cross-helicity $\int_{\mathbb{T}^3}v\cdot b~dx$, respectively. It is worth remarking that without the application of Els\"asser variables equations and K\'arm\'an-Howarth  type equation for the primitive  MHD system \eqref{MHD}, the derivation of \eqref{MHDYaglom4/3lawprimitive variables} was given by Podesta in \cite{[Podesta]}.

Motivated by the works \cite{[DR],[Eyink1]}, the first objective of this paper is to  establish the  Yaglom's $4/3$ law in  various  hydrodynamic equations in the sense as \eqref{DR4/3law}. Now, we state our first result concerning the Yaglom's law for the temperature equation as follows.
 \begin{theorem}\label{the1.1}
 Let  the pair
 $(v,\theta)$ be a weak solution of the temperature equation
 \be\label{temperature equation}
 \theta_{t}+v\cdot\nabla \theta =0,~\text{div\,} v=0.
 \ee Assume that for any $1<p,q,m,n<\infty$ with  $\f2p+\f1m=1,\f2q+\f1n=1 $ such that $(\theta, v)$ satisfies
 \be\label{the1.1c}
\theta \in L^{\infty}(0,T;L^{2}(\mathbb{T}^{3}))\cap L^{p}(0,T;L^{q}(\mathbb{T}^{3}))\ \text{and}\ v\in L^{m}(0,T;L^{n}(\mathbb{T}^{3})). \ee
  Then the function
 $D_{\varepsilon}(v,\theta)=-\f14\int_{\mathbb{T}^{3}}\nabla\varphi_{\varepsilon}(\ell)\cdot\delta v(\ell)|\delta  \theta(\ell)|^{2}d\ell$ converges to a distribution $D(\theta,v)$ in the sense of distributions as $\varepsilon\rightarrow0$, and $D(\theta,v)$ satisfies the local equation of energy
 $$ \partial_{t}(\f12|\theta|^{2})  +\text{div}(\f12 v\theta^{2})=D(v,\theta),
$$
in the sense of distributions.
Moreover, there holds the following $4/3$ law
\be\label{DRYaglom4/3law}
S(v,\theta,\theta)=-\f43D(v,\theta),
\ee
where $S(v,\theta,\theta)=-\lim\limits_{\lambda\rightarrow0}S (v,\theta,\theta;\lambda)=-
\lim\limits_{\lambda\rightarrow0}\f{1}{\lambda}\int_{\partial B }\ell\cdot\delta v (\lambda\ell)|\delta \theta(\lambda\ell)|^{2}\f{d\sigma(\ell) }{4\pi}$.
 \end{theorem}
 \begin{remark}
It is shown in Theorem \ref{the1.1} that the Yaglom's law \eqref{Yaglom4/3law} is valid in the sense of Duchon-Robert \cite{[DR]} and Eyink \cite{[Eyink1]} for the temperature equation. Similar results for the viscous temperature  equation can also be proved. It is worth remarking that an analogue of the   Yaglom's relation \eqref{DRYaglom4/3law} for the quasi-geostrophic equation also holds. We leave this to the interested readers.
 \end{remark}
Next, we consider the Yaglom's $4/3$ law for the  magnetohydrodynamic turbulence  in the following theorem.
  \begin{theorem}\label{the1.2}
 Let $(u,h)\in  L^{\infty}(0,T;L^{2}(\mathbb{T}^{3}))\cap L^{3}(0,T;L^{3}(\mathbb{T}^{3}))$   be a weak solution to the inviscid MHD equations in terms of Els\"asser variables
 \be\left\{\ba\label{mhdElsasser}
&u_{t}+ h\cdot\nabla u+\nabla\Pi =0, \\
&h_{t}+u\cdot\nabla h +\nabla\Pi =0, \\
&\Div u=\Div h=0.
 \ea\right.\ee
  Then the functions
 $D_{\varepsilon}( u,h)=-\f14\int_{\mathbb{T}^{3}}\nabla\varphi_{\varepsilon}(\ell)\cdot\delta u(\ell)|\delta  h(\ell)|^{2}d\ell$ and $D_{\varepsilon}(h,u)=-\f14\int_{\mathbb{T}^{3}}\nabla\varphi_{\varepsilon}(\ell)\cdot\delta h(\ell)|\delta  u(\ell)|^{2}d\ell$ converge respectively to $D(u,h)$ and $D( h,u) $ in the sense of distributions as $\varepsilon\rightarrow0$, and $D(u,h)$ and $D(h,u)$ satisfy the local equation of energy
 $$ \ba
  &\partial_{t}(\f12|u|^{2})  +\text{div}\B[h(\f12u ^{2}+\Pi)\B] = D(u,h),\\
   &\partial_{t}(\f12|h|^{2})  +\text{div}\B[u(\f12 h ^{2}+\Pi)\B] =D(h,u),
\ea$$
in the sense of distributions.
Moreover, there holds
\be\label{DRYaglom4/3lawmhd}
S(u,h,h)=-\f43D(u,h),~~
S(h,u,u)=-\f43D( h,u),\ee
where $S(u,h,h)=-\lim\limits_{\lambda\rightarrow0}S ( u,h,h;\lambda)=-
\lim\limits_{\lambda\rightarrow0}\f{1}{\lambda}\int_{\partial B } \ell \cdot\delta u (\lambda\ell)|\delta h(\lambda\ell)|^{2} \f{d\sigma(\ell) }{4\pi}$ and $S( h,u,u)=-\lim\limits_{\lambda\rightarrow0}S (h, u,u;\lambda)=-
\lim\limits_{\lambda\rightarrow0}\f{1}{\lambda}\int_{\partial B } \ell \cdot\delta h(\lambda\ell)|\delta u(\lambda\ell)|^{2}\f{d\sigma(\ell) }{4\pi}$.
 \end{theorem}
 \begin{remark}
It should be noted that this theorem is consistent with  the result \eqref{MHDYaglom4/3lawElsasser variables} by Politano and Pouquent  in \cite{[PP1],[PP2]}.
 \end{remark}
 By means of
 the relationship \eqref{Elsasser} between the  Els\"asser variables  $u ,h $   and the primitive variables $v,b$ in the MHD equations, we  immediately deduce from Theorem \ref{the1.2} that
 \begin{coro}\label{coro1.3}
 Let   $D_{E}( v,b )=\lim\limits_{\varepsilon\rightarrow0}D_{E}( v,b,\varepsilon)$,
  $D_{CH}( v,b )=\lim\limits_{\varepsilon\rightarrow0}D_{CH}( v,h,\varepsilon)$, where
 \be\ba\label{1.14}
 D_{E}( v,b;\varepsilon)=-\f14\int_{\mathbb{T}^{3}}\nabla\varphi_{\varepsilon}(\ell)\cdot\delta v(\ell)(|\delta  v(\ell) |^{2} +|\delta  b(\ell)|^{2})d\ell+\f12\int_{\mathbb{T}^{3}}\nabla\varphi_{\varepsilon}(\ell)\cdot\delta b(\ell)|\delta  v(\ell)\cdot\delta  b(\ell)| d\ell,\ea\ee
 $$D_{CH}( v,b;\varepsilon)=-\f12\int_{\mathbb{T}^{3}}\nabla\varphi_{\varepsilon}(\ell)\cdot\delta v(\ell) |\delta  v(\ell)\delta  b(\ell)| d\ell+\f14\int_{\mathbb{T}^{3}}\nabla\varphi_{\varepsilon}(\ell)\cdot\delta b(\ell)(| \delta  v(\ell)|^{2}|+| \delta  b(\ell)|^{2}|) d\ell,$$ and
\be\ba \label{1.15}&S_{1}(v,v,v)=-\lim\limits_{\lambda\rightarrow0}S_{1} ( v,v,v;\lambda)=-
\lim\limits_{\lambda\rightarrow0}\f{1}{\lambda}\int_{\partial B } \ell \cdot\delta v (\lambda\ell)|\delta v(\lambda\ell)|^{2}\f{d\sigma(\ell) }{4\pi},\\
&S_{2}(v,b,b)=-\lim\limits_{\lambda\rightarrow0}S_{2} ( v,b,b;\lambda)=-
\lim\limits_{\lambda\rightarrow0}\f{1}{\lambda}\int_{\partial B } \ell \cdot\delta v (\lambda\ell)|\delta b(\lambda\ell)|^{2}\f{d\sigma(\ell) }{4\pi},
\\
&S_{3}(b,v,b)=-\lim\limits_{\lambda\rightarrow0}S _{3}(b,v,b;\lambda)=-
\lim\limits_{\lambda\rightarrow0}\f{1}{\lambda}\int_{\partial B } \ell \cdot\delta b (\lambda\ell)|\delta v(\lambda\ell)\cdot\delta b(\ell)| \f{d\sigma(\ell) }{4\pi},\\
&S_{4}(v,v,b)=-\lim\limits_{\lambda\rightarrow0}S _{4}( v,v,b;\lambda)=-
\lim\limits_{\lambda\rightarrow0}\f{1}{\lambda}\int_{\partial B } \ell \cdot\delta v (\lambda\ell)|\delta v(\lambda\ell)\cdot\delta b(\lambda\ell)| \f{d\sigma(\ell) }{4\pi},\\
&S_{5}(b,v,v)=-\lim\limits_{\lambda\rightarrow0}S_{5} ( b,v,v;\lambda)=-
\lim\limits_{\lambda\rightarrow0}\f{1}{\lambda}\int_{\partial B } \ell \cdot\delta b (\lambda\ell)|\delta v(\lambda\ell)|^{2}\f{d\sigma(\ell) }{4\pi},\\
&S_{6}(b,b,b)=-\lim\limits_{\lambda\rightarrow0}S_{6} ( b,b,b;\lambda)=-
\lim\limits_{\lambda\rightarrow0}\f{1}{\lambda}\int_{\partial B } \ell \cdot\delta b (\lambda\ell)|\delta b(\lambda\ell)|^{2}\f{d\sigma(\ell) }{4\pi}.\ea\ee
Then, there holds
\be\ba\label{1.16}
&S_{1}(v,v,v)+S_{2}(v,b,b)-2 S_{3}(b, v,b )=-\f43D_{E}( v,b ),\\&
2S_{4}( v,v,b )-S_{5}(  b,v,v )- S_{6}(b,b,b )=-\f43D_{CH}( v,b ).\ea\ee
 \end{coro}
 It is worth remarking that we can directly obtain \eqref{1.16} by an application of the original MHD equations as follows.
 \begin{theorem}\label{the1.4}
Suppose that $(v, b)\in L^{\infty}(0,T;L^{2}(\mathbb{T}^{3}))\cap L^{3}(0,T;L^{3}(\mathbb{T}^{3}))$   be a weak solution of the MHD equations
 \be\left\{\ba\label{MHD}
&v_{t}+v\cdot\nabla v-b\cdot\nabla b+\nabla\Pi =0, \\
&b_{t}+v\cdot\nabla b-b\cdot\nabla v =0, \\
&\Div v=\Div b=0.
 \ea\right.\ee
 Here the scalar function $\Pi=\pi+\f{1}{2}b^{2}$ represents the total pressure.
 Let  $ D_{E}( v,b,\varepsilon),$ $ D_{CH}( v,b,\varepsilon)$  and        $S_{i}( v,b ) $   ($i=1,2,\cdots,6$) be defined in
 \eqref{1.14} and \eqref{1.15}.
  Then the functions $D_{E}( v,b,\varepsilon)$
 and $D_{CH}( v,b,\varepsilon) $ converge respectively to  $D_{E}( v,b )$ and $D_{CH}( v,b)$   in the sense of distributions as $\varepsilon\rightarrow0$, and   $D_{E}( v,b )$ and $D_{CH}( v,b)$ satisfy the local equation
\be \ba\label{1.18}
   &\partial_{t}(\f{v^{2}+b^{2}}{2}  )  +\text{div}\B[v\B( \f12(|v|^{2}+|b |^{2})+\Pi\B)-b(b\cdot v)\B]=D_{E}( v,b ),\\
  &\partial_{t}(  v\cdot b)  +\text{div}\B[v(v\cdot b)-\f12(|b|^{2})b-\f12bv^{2}+b\Pi\B]= D_{CH}( v,b ),
\ea\ee
in the sense of distributions.
Moreover, there holds \eqref{1.16}.
 \end{theorem}
We notice that $\eqref{1.18}_{1}$ and $\eqref{1.16}_{1}$ have been proved by Guo-Tan-Wu in \cite{[GTW]} and Guo-Tan in \cite{[GT]}, respectively. Hence, we shall mainly focus our attention on the proof of results involving cross-helicity in this theorem. Corollary \ref{coro1.3} reproduces the corresponding result in \cite{[GTW]} for the inviscid MHD equations in terms of Els\"asser variables. In history, after discovering the cross-helicity conservation, Moffatt in \cite{[Moffatt]} first introduced the helicity law in an inviscid fluid. Moreover, Gomez-Plitano-Pouquet in \cite{[GPP]} gave the third-order structure function relation for the helicity of the Euler equations \eqref{euler}  as
\be\label{1.20}
   \langle\delta v_{L}(r)|\delta v(r)\cdot\delta \omega(r)| \rangle -\f12\langle\delta \omega_{L}(r)|\delta v(r) |^{2}  \rangle=-\f43\bar{\epsilon}  r,
\ee
 where $\bar{\epsilon}$ represents the helicity dissipation rates. Based on this, our third result is concerned with the Yaglom's law for the helicity of  ideal incompressible Euler equations. Below is our corresponding result.
  \begin{theorem}\label{the1.5}
 Let $\omega\in L^{\infty }(0,T;L^{\f{3}{2}}(\mathbb{T}^{3}))$ and $(v,\omega)$   be a weak solution of the Euler equations
 \be\left\{\ba\label{euler}
&v_{t}+ v\cdot\nabla v+\nabla\pi =0, \\
& \omega_{t}+v\cdot\nabla \omega-\omega\cdot\nabla v =0, \\
&\Div v=\Div \omega=0.
 \ea\right.\ee
Suppose  that for any $1<m,n<\infty,3\leq p,q<\infty$ with $\f2p+\f1m=1, \f2q+\f1n=1$ such that $(v,\omega)$  meets
 \be\label{the1.5c}
v\in L^{p}(0,T;L^{q}(\mathbb{T}^{3}))\ \text{and}\ \omega\in L^{m}(0,T;L^{n}(\mathbb{T}^{3})).  \ee
  Then the function
 $$D_{\varepsilon}( v,\omega)=-\f12\int_{\mathbb{T}^{3}}\nabla\varphi_{\varepsilon}(\ell)\cdot\delta v(\ell)|\delta  \omega(\ell)\cdot\delta  v(\ell)| d\ell+\f14\int_{\mathbb{T}^{3}}\nabla\varphi_{\varepsilon}(\ell)\cdot\delta \omega(\ell)|\delta  v(\ell)|^{2}d\ell$$
 converges to a distribution $D(v, \omega) $ in the sense of distributions as $\varepsilon\rightarrow0$, and $D( v,\omega)$ satisfies the local equation of helicity
 $$ \ba
  &\partial_{t}( v\cdot\omega)  +\text{div}\B[v(\omega\cdot v)-\f12\omega(|v|^{2})+\omega\pi\B]= D(v,\omega),
\ea$$
in the sense of distributions.
Moreover, there holds
\be\label{helicityYaglom4/3law}
2S_{7}(v,\omega,v)-S_{8}(\omega, v,v)=-\f43D(v,\omega),
 \ee
where $S_{7}(v,\omega,v)=-\lim\limits_{\lambda\rightarrow0}S_{7}(v,\omega,v;\lambda)=-
\lim\limits_{\lambda\rightarrow0}\f{1}{\lambda}\int_{\partial B } \ell \cdot\delta v (\ell)|\delta \omega(\lambda\ell)\cdot\delta v(\lambda\ell)| \f{d\sigma(\ell) }{4\pi}$ and $S_{8}(  \omega,v,v)=-\lim\limits_{\lambda\rightarrow0}S_{8}( \omega,v,v;\lambda)
=-\lim\limits_{\lambda\rightarrow0}\f{1}{\lambda}\int_{\partial B } \ell \cdot\delta \omega (\lambda\ell)|\delta v(\lambda\ell)|^{2}\f{d\sigma(\ell) }{4\pi}$.
 \end{theorem}
\begin{remark}
It seems that there exists a little gap between \eqref{1.20} and \eqref{helicityYaglom4/3law}.
\end{remark}
\begin{remark}
A well known helicity criterion for the incomressible Euler equations obtained in \cite{[CCFS]} is that the helicity is conserved provided that  $v\in L^3(0,T;B^{\f23}_{3,q^{\natural}})$ with $q^{\natural}<\infty.$ Hence, if $m\geq3$ in \eqref{the1.5c}, we require $n<9/4$ in this theorem. Since a special case of \eqref{the1.5c} is $p=m=3$, $q=\f92$ and $n=\f95$, the condition \eqref{the1.5c} in no empty.
\end{remark}

We would like to point out that the dissipation term \eqref{drKHMr} introduced by  Duchon-Robert in \cite{[DR]} not only can be applied to the deduction of the $4/3$ law and $4/5$ law, but also  may be used in the study of  the  Onsager's conjecture (see \cite{[DR]}). Hence, it seems that there are some potential applications of the above theorems in the research of conserved quantities in fluid mechanics. In particular, the study of conservation of energy, cross-helicity and helicity in  incompressible fluid has attracted a lot of attention (see e.g. \cite{[Chae],[Chae1],[CET],[CY],[CCFS],[WWY],[Eyink0],[EGSW],[DE],[WZ],[BGSTW],[CKS]}). We shall take Theorem \ref{the1.5} as an example to illustrate this application.
 \begin{coro}\label{coro1.6}
 We use the notations in Theorem \ref{the1.5}.
Assume that $v$ and $\omega$ satisfy
\be\label{1.23}
\ba
&\B(\int_{\mathbb{T}^{3}}|v(x+\ell,t)-v(x,t)|^{\f{9}{2}}dx\B)^{\f{2}{9}}\leq C(t)^{\f{1}{r_{1}}}|\ell|^{\alpha}\sigma^{\f13}(\ell),\\ &\B(\int_{\mathbb{T}^{3}}|\omega(x+\ell,t)-\omega(x,t)|^{\f95}dx\B)^{\f59}\leq C(t)^{\f{1}{r_{2}}}|\ell|^{\beta}\sigma^{\f13}(\ell),\\
&\text{with}\ \f{2}{r_{1}}+\f{1}{r_{2}}=1, 1<r_{1},r_{2}<\infty, 2\alpha+\beta\geq1,
\ea\ee
where both of $C_{i}(t)$ for $i=1,2$  are integrable functions on $[0,T]$, and $\sigma_{i}(\ell)$ for $i=1,2$  are both bounded functions on some neighborhood of the origin. Suppose that at least one of  $\sigma_{i}(\ell)$ obeys $\sigma_{i}(\ell) \rightarrow0$ as $\ell \rightarrow0$. Then the helicity is conserved.
 \end{coro}
\begin{remark}
 This corollary is an improvement of corresponding results in \cite{[Chae]}.
\end{remark}
We turn our attention back to the Yaglom's   law.
Almost all above results are in good agreement with the known results derived from K\'arm\'an-Howarth type equations \eqref{KHMrelation}.
This
suggest that the Yaglom's   law of the fluid equations holds if an analogue of dissipation term as Duchon-Robert is obtained. It is helpful to discover the new Yaglom's law in  turbulence.  Next, we establish     the new Yaglom's law in the   Oldroyd-B model
and for the subgrid scale $\alpha$-models of turbulence.
Indeed, we are conserved with four-three relation of
the following simplified version of inviscid Oldroyd-B equations introduced by  Elgindi and   Rousset in \cite{[ER]}
\be\left\{\ba\label{Oldroyd-B3}
&v_{t}  +v\cdot\nabla v +\nabla \pi=  \Div\tau, \\
&\tau_{t} + v\cdot\nabla \tau + \tau =  \mathcal{D}(v)=\f12(\nabla v+\nabla v^{\text{T}}),\\
&\Div v=0.
\ea\right.\ee
Here $v(t,x)$ stands for the velocity, $\pi(t,x)$ is the pressure and $\tau(t,x)$ is
 the non-Newtonian part of the stress tensor which is a $(d,d)$ symmetric matrix.
 The Oldroyd-B model describes the evolution of some viscoelastic flows(see e.g.\cite{[ER],WZ,[LLZ],[ZYW]}).
 \begin{theorem}\label{the1.6}
 Let  the pair
 $(v,\tau)\in L^{\infty}(0,T;L^{2}(\mathbb{T}^{3}))\cap L^{3}(0,T;L^{3}(\mathbb{T}^{3}))$ be a weak solution of  simple inviscid Oldroyd-B equations \eqref{Oldroyd-B3}.
   Then the function
 $$D_{\varepsilon}(v,\tau)=-\f14\int_{\mathbb{T}^{3}}\nabla\varphi_{\varepsilon}(\ell)\cdot\delta v(\ell)|\delta  v(\ell)|^{2}d\ell-\f14\int_{\mathbb{T}^{3}}\nabla\varphi_{\varepsilon}(\ell)\cdot\delta v(\ell)|\delta \tau(\ell)|^{2}d\ell$$ converges to a distribution $D(v,\tau)$ in the sense of distributions as $\varepsilon\rightarrow0$, and $D(v,\tau)$ satisfies the local equation of energy
 $$ \partial_{t}(\f{|v|^{2}+|\tau|^{2}}{2})  +\text{div}\B[\f{|v|^{2}+|\tau|^{2}}{2} v+v\pi\B]-\partial_{i}(\tau_{ij}v_{j})+|\tau|^{2}=D( v,\tau),
$$
in the sense of distributions on $(0,T)\times \mathbb{T}^3$.
Moreover, there holds
\be\label{DRYaglom4/3lawOldroyd-B}
S(v,v,v)+S(v,\tau,\tau)=-\f43D( v, \tau),
\ee
where $S(v,v,v)=-\lim\limits_{\lambda\rightarrow0}S (v,v,v;\lambda)=-
\lim\limits_{\lambda\rightarrow0}\f{1}{\lambda}\int_{\partial B }\ell\cdot\delta v (\lambda\ell)|\delta v(\lambda\ell)|^{2}\f{d\sigma(\ell) }{4\pi}$ and
$S(v,\tau,\tau)=-\lim\limits_{\lambda\rightarrow0}S (v,\tau,\tau;\lambda)=-
\lim\limits_{\lambda\rightarrow0}\f{1}{\lambda}\int_{\partial B }\ell\cdot\delta v (\lambda\ell)|\delta \tau(\lambda\ell)|^{2}\f{d\sigma(\ell) }{4\pi}$.
 \end{theorem}
 \begin{remark}
 This is the first result involving the Yaglom's law to the Oldroyd-B models.
It is interesting to obtain   four-three law  as  \eqref{DRYaglom4/3lawOldroyd-B} in the general Oldroyd-B model.
Along the same line of \eqref{coro1.6}, one can invoke the dissipation term to consider the   energy conservation for the weak solutions in the Oldroyd-B equations  \eqref{Oldroyd-B3}.
 \end{remark}

Very recently, Boutros and  Titi \cite{[BT]}  studied the   Onsager's conjecture for subgrid scale $\alpha$-models of
turbulence including Leray-$\alpha$ model \eqref{Laerfa},  Euler-$\alpha$ model \eqref{Euleraerfa}, modified Leray-$\alpha$ model \eqref{mLaerfa},  Clark-$\alpha$ model
\eqref{Caerfa} and   Leray-$\alpha$ MHD equations
 \eqref{LMHD}. Notice that all the proofs in   \cite{[BT]} rely on the corresponding dissipation term
as Duchon-Robert \eqref{drKHMr}. As  aforementioned, in the spirt of Theorem  \ref{the1.1}-\ref{the1.4}, we  immediately get the following
 Yaglom's relations for subgrid scale $\alpha$-models.
 \begin{theorem}\label{the1.7}
 We assume that $v_{k}=(1-\alpha^{2}\Delta) u_{k}$ and $\partial_{i}u_{i}=\partial_{j}v_{j}$ for $i,j,k\in\{1,2,3\}.$\\
 (1) Let $$S(u,v,v)=-\lim\limits_{\lambda\rightarrow0}S (u,v,v;\lambda)=-
\lim\limits_{\lambda\rightarrow0}\f{1}{\lambda}\int_{\partial B }\ell\cdot\delta u (\lambda\ell)|\delta v(\lambda\ell)|^{2}\f{d\sigma(\ell) }{4\pi}$$ and $D_{1}(u,v)$ satisfy the local equation of energy
\be\label{Laerfa} \partial_{t}(\f{|v|^{2} }{2})  +\text{div}\B[\f{|v|^{2} }{2} v+v\pi\B]=D_{1}( u,v),\ee
 in the sense of distributions  for the weak solutions $v$ of Leray-$\alpha$ model $${v_{j}}_{t}+ \partial_{i}(u_{i}v_{j})+\partial_{j}\Pi =0.$$
Moreover, there holds
$$S(u,v,v)=-\f43 D_{1}( u,v).$$
(2) Let
$$\ba
&S_{1}(u,u,u)=-\lim\limits_{\lambda\rightarrow0}S (u,u,u;\lambda)=-
\lim\limits_{\lambda\rightarrow0}\f{1}{\lambda}\int_{\partial B }\ell\cdot\delta u (\lambda\ell)|\delta u(\lambda\ell)|^{2}\f{d\sigma(\ell) }{4\pi},\\
&S_{2}(u,\partial_{k}u_{j},\partial_{k}u_{j})=-\lim\limits_{\lambda\rightarrow0}S (u,\partial_{k}u_{j},\partial_{k}u_{j};\lambda)=-
\lim\limits_{\lambda\rightarrow0}\f{1}{\lambda}\int_{\partial B }\ell\cdot\delta u (\lambda\ell)|\delta \partial_{k}u_{j}(\lambda\ell)|^{2}\f{d\sigma(\ell) }{4\pi},
\ea
$$ and $D_{2}(u)$ satisfy the local equation of energy
$$\ba &\partial_{t}(\f{|u|^{2} }{2})+\alpha^{2}\partial_{t}(\f{|\nabla v|^{2} }{2}) -\alpha^{2}\partial_{t}\partial_{i}(u_{j}\partial_{i}u_{j}) +\alpha^{2}\partial_{i}(\partial_{t}u_{j}\partial_{i}u_{j})+\alpha^{2}
\text{div}(\partial_{k}u_{j}u_{j}\partial_{k}u)\\&+\alpha^{2}
\text{div}(\partial_{k}u_{j}\partial_{k}u_{j}u)+\alpha^{2}
\partial_{i}\partial_{k}(u_{i}\partial_{k}u_{j}u_{j})+\text{div}\B[\f{|u|^{2} }{2} u+u\pi\B]=D_{2}(u),\ea$$
 in the sense of distributions  for the weak solutions $u$ of Euler-$\alpha$ model
\be\label{Euleraerfa}
{v_{j}}_{t}+  \partial_{i}(u_{i}v_{j})+v_{i}\partial_{j}u_{i}+\partial_{j}\Pi =0.\ee
Moreover, we have
$$S_{1}(u,u,u)+2\alpha^{2}S_{2}(u,\partial_{k}u_{j},\partial_{k}u_{j})=-\f43 D_{2}( u ).$$
(3) Let
$$\ba
&S_{1}(u,u,u)=-\lim\limits_{\lambda\rightarrow0}S (u,u,u;\lambda)=-
\lim\limits_{\lambda\rightarrow0}\f{1}{\lambda}\int_{\partial B }\ell\cdot\delta u (\lambda\ell)|\delta u(\lambda\ell)|^{2}\f{d\sigma(\ell) }{4\pi},\\
&S_{2}(u,\partial_{k}u_{i},\partial_{k}u_{j})=-\lim\limits_{\lambda\rightarrow0}S (u,\partial_{k}u_{i},\partial_{k}u_{j};\lambda)=-
\lim\limits_{\lambda\rightarrow0}\f{1}{\lambda}\int_{\partial B }\ell_{i} \delta u_{j} (\lambda\ell) (\delta \partial_{k}u_{i}(\lambda\ell) \delta \partial_{k}u_{j}(\lambda\ell) ) \f{d\sigma(\ell) }{4\pi},
\ea
$$ and $D_{3}(u)$ satisfy the local equation of energy
$$\ba &\partial_{t}(\f{|u|^{2} }{2})+\alpha^{2}\partial_{t}(\f{|\nabla v|^{2} }{2}) -\alpha^{2}\partial_{t}\partial_{i}(u_{j}\partial_{i}u_{j}) +\alpha^{2}\partial_{i}(\partial_{t}u_{j}\partial_{i}u_{j})+\alpha^{2}
\text{div}(\partial_{k}u_{j}u_{j}\partial_{k}u),\\&+\f{\alpha^{2}}{2}
\text{div}(\partial_{k}|u|^{2}\partial_{k} u)+\text{div}\B[\f{|u|^{2} }{2} u+u\pi\B]=D_{3}(u),\ea$$
in the sense of distributions  for the weak solutions $u$ of modified Leray-$\alpha$ model \be\label{mLaerfa}
{v_{j}}_{t}+  \partial_{i}(v_{i} u_{j}) +\partial_{j}\Pi =0.\ee
Moreover, there holds
 \be\label{mLaerfalaw}
 S_{1}(u,u,u)+2\alpha^{2}S_{2}(u,\partial_{k}u_{j},\partial_{k}u_{j})=-\f43 D_{3}( u).\ee

(4)Let
$$\ba
&S_{1}(u,u,u)=-\lim\limits_{\lambda\rightarrow0}S (u,u,u;\lambda)=-
\lim\limits_{\lambda\rightarrow0}\f{1}{\lambda}\int_{\partial B }\ell\cdot\delta u (\lambda\ell)|\delta u(\lambda\ell)|^{2}\f{d\sigma(\ell) }{4\pi},\\
&S_{2}(u,\partial_{k}u_{i},\partial_{k}u_{j})=-\lim\limits_{\lambda\rightarrow0}S (u,\partial_{k}u_{i},\partial_{k}u_{j};\lambda)=-
\lim\limits_{\lambda\rightarrow0}\f{1}{\lambda}\int_{\partial B }\ell_{i} \delta u_{j} (\lambda\ell) (\delta \partial_{k}u_{i}(\lambda\ell) \delta \partial_{k}u_{j}(\lambda\ell) ) \f{d\sigma(\ell) }{4\pi},\\&S_{3}(u,\partial_{k}u_{j},\partial_{k}u_{j})=-\lim\limits_{\lambda\rightarrow0}S (u,\partial_{k}u_{j},\partial_{k}u_{j};\lambda)=-
\lim\limits_{\lambda\rightarrow0}\f{1}{\lambda}\int_{\partial B }\ell\cdot\delta u  (\lambda\ell) | \delta \partial_{k}u_{j}(\lambda\ell) |^{2} \f{d\sigma(\ell) }{4\pi},
\ea
$$ and $D_{4}(u)$ satisfies the local equation of energy
$$\ba &\partial_{t}(\f{|u|^{2} }{2})+\alpha^{2}\partial_{t}(\f{|\nabla v|^{2} }{2}) -\alpha^{2}\partial_{t}\partial_{i}(u_{j}\partial_{i}u_{j}) +\alpha^{2}\partial_{i}(\partial_{t}u_{j}\partial_{i}u_{j})+\alpha^{2}
\partial_{i}(u_{j}\partial_{k}u_{j}\partial_{k}u_{i})\\&+\f{\alpha^{2}}{2}
 \partial_{i}\partial_{k}(u_{k}\partial_{k}u_{i}u_{j})
 +\f{3\alpha^{2}}{2}\partial_{k}(\partial_{i}u_{j}\partial_{k}u_{i}u_{j})+\text{div}\B[\f{|u|^{2} }{2} u+u\pi\B]=D_{4}(u).\ea$$
in the sense of distributions  for the weak solutions $u$ of the Clark-$\alpha$ model
\be\label{Caerfa}
{v_{j}}_{t}+  \partial_{i}(u_{i} v_{j})+  \partial_{i}(u_{j} v_{i})-\partial_{i}(u_{i} u_{j})
-\alpha^{2}\partial_{i}[\partial_{k}u_{i}\partial_{k}u_{i}] +\partial_{j}\Pi =0.\ee
Moreover, it is valid that
$$S_{1}(u,u,u)+2\alpha^{2}S_{2}(u,\partial_{k}u_{j},\partial_{k}u_{j})+\alpha^{2}S_{3}=-\f43 D_{4}( u,v).$$
(5) Let
$$\ba
&S_{1}(u,v,v)=-\lim\limits_{\lambda\rightarrow0}S_{1} ( u,v,v,\lambda)=-
\lim\limits_{\lambda\rightarrow0}\f{1}{\lambda}\int_{\partial B } \ell \cdot\delta u (\lambda\ell)|\delta v(\lambda\ell)|^{2}\f{d\sigma(\ell) }{4\pi},\\
&S_{2}(u,H,H)=-\lim\limits_{\lambda\rightarrow0}S_{2} ( u,H,H,\lambda)=-
\lim\limits_{\lambda\rightarrow0}\f{1}{\lambda}\int_{\partial B } \ell \cdot\delta u (\lambda\ell)|\delta H(\lambda\ell)|^{2}\f{d\sigma(\ell) }{4\pi},
\\
&S_{3}(H,v,H)=-\lim\limits_{\lambda\rightarrow0}S _{3}( H,v,H,\lambda)=-
\lim\limits_{\lambda\rightarrow0}\f{1}{\lambda}\int_{\partial B } \ell \cdot\delta H (\lambda\ell)|\delta v(\lambda\ell)\cdot\delta H(\ell)| \f{d\sigma(\ell) }{4\pi},
\ea
$$ and $D_{5}( v,u,H )$ satisfy the local equation of energy
$$\ba
 \partial_{t}(\f{v^{2}+H^{2}}{2}  )  +\text{div}\B[( \f12(|v|^{2}+|H |^{2})u+v\Pi)-H(H\cdot v)\B]=D_{5}( v,u,H ),
\ea$$
 in the sense of distributions  for the weak solutions $v,H$ of the   Leray-$\alpha$ MHD equations
 \be\left\{\ba\label{LMHD}
&v_{t}+u\cdot\nabla v-H\cdot\nabla H+\nabla\Pi =0, \\
&H_{t}+u\cdot\nabla H-H\cdot\nabla v =0, \\
&\Div v=\Div H=0.
 \ea\right.\ee
Moreover, we have
$$S_{1}(v,v,v)+S_{2}(v,H,H)-2 S_{3}(H, v,H )=-\f43D_{5}( v,u,H ).$$
  \end{theorem}
\begin{remark}
Though the K\'arm\'an-Howarth equation of the  Lagrangian averaged Euler equations  \eqref{Euleraerfa} and its $2/15$ law  were  obtained  by Holm in \cite{[Holm]},   it seems that even  for this system the $3/4$ law \eqref{mLaerfalaw}  does not exist in the known literature. To the knowledge of the authors, all the results in Theorem \ref{the1.7} are novel.
   \end{remark}  \begin{remark}
When $\alpha=0$,  most the $4/3$ laws in this theorem reduce  to  \eqref{DR4/3law}
  \end{remark}
 Notice that there exist two $3/4$ laws  \eqref{1.16} for the  energy and cross-helicity in the  standard MHD equations. A natural equation is whether the Yaglom's law  for the cross-helicity  of the Leray-$\alpha$ MHD equations is still valid. The last theorem is devoted to this.
\begin{theorem}\label{the1.8}
Suppose that $(v, H )\in L^{\infty}(0,T;L^{2}(\mathbb{T}^{3}))\cap \in L^{3}(0,T;L^{3}(\mathbb{T}^{3}))$   be a weak solution of the  \eqref{LMHD}.
 Here the scalar function $\Pi=\pi+\f{1}{2}b^{2}$ represents the total pressure.
 Let    $$  D_{C}( v,u,H,\varepsilon)=-\f12\int_{\mathbb{T}^{3}}\nabla\varphi_{\varepsilon}(\ell)\cdot\delta u(\ell) |\delta  v(\ell)\delta  H(\ell)| d\ell+\f14\int_{\mathbb{T}^{3}}\nabla\varphi_{\varepsilon}(\ell)\cdot\delta H(\ell)(| \delta  v(\ell)|^{2}|+| \delta  H(\ell)|^{2}|) d\ell,$$  and
    $$\ba
& S_{4}(u,v,H)=-\lim\limits_{\lambda\rightarrow0}S _{4}( u,v,H,\lambda)=-
\lim\limits_{\lambda\rightarrow0}\f{1}{\lambda}\int_{\partial B } \ell \cdot\delta u (\lambda\ell)|\delta v(\lambda\ell)\cdot\delta H(\lambda\ell)| \f{d\sigma(\ell) }{4\pi},\\
&S_{5}(H,v,v)=-\lim\limits_{\lambda\rightarrow0}S_{5} ( H,v,v;\lambda)=-
\lim\limits_{\lambda\rightarrow0}\f{1}{\lambda}\int_{\partial B } \ell \cdot\delta H (\lambda\ell)|\delta v(\lambda\ell)|^{2}\f{d\sigma(\ell) }{4\pi},\\
&S_{6}(H,H,H)=-\lim\limits_{\lambda\rightarrow0}S_{6} ( H,H,H,\lambda)=-
\lim\limits_{\lambda\rightarrow0}\f{1}{\lambda}\int_{\partial B } \ell \cdot\delta H (\lambda\ell)|\delta H(\lambda\ell)|^{2}\f{d\sigma(\ell) }{4\pi}.\ea$$
  Then the functions   $D_{C}( v,H,\varepsilon) $ converge to    $D_{C}( v,H)$   in the sense of distributions as $\varepsilon\rightarrow0$, and     $D_{C}( v,H)$ satisfy the local equation
\be \ba\label{1.29}
  &\partial_{t}(  v\cdot H)  +\text{div}\B[u(v\cdot H)-\f12(|H|^{2})H-\f12Hv^{2}+H\Pi\B]= D_{C}( v,u,H ),
\ea\ee
in the sense of distributions.
Moreover, there holds
$$2S_{4}( u,v,H )-S_{5}(H,v,v )- S_{6}(H,H,H )=-\f43D_{C}( v,u,H ).$$
 \end{theorem}
\begin{remark}
Compared with theorem \ref{the1.7}, this theorem is self-contained and independent of the work \cite{[BT]}.
\end{remark}
\begin{remark}
As Corollary \ref{coro1.6}, the dissipation term  can be used for the  investigation of the cross-helicity conservation of the weak solutions in Leray-$\alpha$ MHD equations.
\end{remark}

To end  this section, we introduce some notations which will be used in this paper.
Firstly, for $p\in [1,\,\infty]$, the notation $L^{p}(0,\,T;X)$ stands for the set of measurable functions $f$ on the interval $(0,\,T)$ with values in $X$ and $\|f\|_{X}$ belonging to $L^{p}(0,\,T)$.
Secondly, we will use the standard mollifier kernel, i.e. $\varphi(x)=C_0e^{-\frac{1}{1-|x|^2}}$ for $|x|<1$ and $\varphi(x)=0$ for $|x|\geq 1$, where $C_0$ is a constant such that $\int_{\mathbb{R}^3}\varphi (x) dx=1$. Eventually, for $\varepsilon>0$, we denote the rescaled mollifier by  $\varphi_{\varepsilon}(x)=\frac{1}{\varepsilon^3}\varphi(\frac{x}{\varepsilon})$, and for any function $f\in L^1_{\textrm{loc}}(\mathbb{R}^3)$, its mollified version is defined by
$$
f^\varepsilon(x)=\int_{\mathbb{R}^3}\varphi_{\varepsilon}(x-y)f(y)dy,\ \ x\in \mathbb{R}^3.
$$

\section{ Yaglom's law of conserved quantities  in turbulence}
This section is  devoted to
 the proof of the Yaglom's $4/3$ law  of   energy, cross-helicity and helicity  in  various turbulence models, i.e. the temperature equations, the inviscid MHD equations and ideal Euler equations. The key tool is the dissipation term owing to the lack of smoothness  of the solutions in fluids, which was first introduced by Duchon-Robert in \cite{[DR]} when studying the Onsager's conjecture of the ideal incompressible Euler equations. Moreover,  at the end of this section, as an application, we will invoke the dissipation term to study the conservation of helicity in the  ideal incompressible Euler equations.
\subsection{$4/3$   law of energy    in the temperature equation  }
To begin with, we will prove the Yaglom's 4/3 law for the temperature equations \eqref{temperature equation} by invoking the dissipation term.
\begin{proof}[Proof of Theorem \ref{the1.1}]
First, mollifying the  equation \eqref{temperature equation} (see the notations in Section 1), one has
$$
\partial_{t}\theta^{\varepsilon}+\partial_{i}(v_{i}\theta)^{\varepsilon}=0.
$$
Second, multiplying the above equation by $\theta$ and the  equation \eqref{temperature equation} by $\theta^{\varepsilon}$, respectively,   then sum them  together to find
 $$
 \partial_{t}(\theta\theta^{\varepsilon})+
\partial_{i}(v_{i}\theta)^{\varepsilon}\theta
+\partial_{i}(v_{i}\theta)\theta^{\varepsilon}=0.
$$
For the last two terms on the LHS of above equation, after performing elementary calculations, we discover that
$$
\partial_{i}(v_{i}\theta)^{\varepsilon}\theta
+\partial_{i}(v_{i}\theta)\theta^{\varepsilon}=\partial_{i}(v_{i}\theta\theta^{\varepsilon})
+\partial_{i}(v_{i}\theta)^{\varepsilon}\theta-
v_{i}\theta\partial_{i}\theta^{\varepsilon},
$$
which yields that
$$ \f12\partial_{t}(\theta\theta^{\varepsilon})+\f12\partial_{i}(v_{i}\theta\theta^{\varepsilon})
=-\f12
\B(\partial_{i}(v_{i}\theta)^{\varepsilon}\theta
-v_{i}\theta\partial_{i}\theta^{\varepsilon}\B) .  $$
This means that $v_{i}\theta\partial_{i}\theta^{\varepsilon}-\partial_{i}(v_{i}\theta)^{\varepsilon}\theta
$ converges in accord with $  \partial_{t}(\theta\theta^{\varepsilon})+ \partial_{i}(v_{i}\theta\theta^{\varepsilon})
$ in the sense of distributions as $\varepsilon\rightarrow0.$
Actually, according to hypothesis \eqref{the1.1c}, we can assert that $\partial_{t}(\theta\theta^{\varepsilon})+ \partial_{i}(v_{i}\theta\theta^{\varepsilon})
$ tend to  $\partial_{t}(\theta\theta )+
\partial_{i}(v_{i}\theta\theta )
$ in the sense of distributions as $\varepsilon\rightarrow0.$
Thus, our next goal is to show that the limit of $-\f14\int_{\mathbb{T}^{3}}
\nabla\varphi_{\varepsilon}(|\ell|)\cdot\delta v(\ell)|\delta  \theta(\ell)|^{2}d\ell$ is the same as that of $-\f12\B(\partial_{i}(v_{i}\theta)^{\varepsilon}\theta
-v_{i}\theta\partial_{i}\theta^{\varepsilon}\B)$ as $\varepsilon\rightarrow0.$
To this end, we write
$$\delta v(\ell)=v_{i}(x+\ell)-v_{i}(x)=V_{i}-v_{i} ~ \text{and}  ~\delta  \theta(\ell)=\theta(x+\ell)-\theta(x)=\Theta -\theta. $$
After a  straightforward calculation, it is easy to check that
$$\ba
&\int_{\mathbb{T}^{3}}\nabla\varphi_{\varepsilon}(\ell)\cdot\delta v(\ell)|\delta  \theta(\ell)|^{2}d\ell\\
=&\int_{\mathbb{T}^{3}} \partial_{l_{i}}\varphi_{\varepsilon}(\ell)[V_{i}(x+\ell)-v_{i}(x)][\Theta(x+\ell)-\theta(x)]^{2} d\ell\\=&\int_{\mathbb{T}^{3}} \partial_{l_{i}}\varphi_{\varepsilon}(\ell)[V_{i} \Theta ^{2}
-2V_{i}\Theta\theta +V_{i}\theta^{2}-v_{i}\Theta^{2}+2v_{i}\Theta\theta-v_{i} \theta^{2}]  d\ell.
\ea$$
For the first term on the RHS of above equation, by changing variables, one has
$$\ba
\int_{\mathbb{T}^{3}} \partial_{l_{i}}\varphi_{\varepsilon}(\ell) V_{i} \Theta ^{2}
   d\ell=&\int_{\mathbb{T}^{3}} \partial_{l_{i}}\varphi_{\varepsilon}(\ell) v_{i}(x+\ell)
   \theta^{2}(x+\ell) d\ell\\
   =&\int_{\mathbb{T}^{3}} \partial_{\eta_{i}}\varphi_{\varepsilon}(\eta-x) v_{i}(\eta)
   \theta^{2}(\eta) d\eta\\
   =&-\int_{\mathbb{T}^{3}} \partial_{x_{i}}\varphi_{\varepsilon}(\eta-x) v_{i}(\eta)
   \theta^{2}(\eta) d\eta\\
   =&-\partial_{i}(v_{i} \theta^{2}\ast\varphi_{\varepsilon} )\\
   =&-\partial_{i}(v_{i} \theta^{2} )^{\varepsilon}.
\ea$$
Then by the same token as  above and using the divergence-free condition, we note that
$$\ba
&\int_{\mathbb{T}^{3}} \partial_{l_{i}}\varphi_{\varepsilon}(\ell)[V_{i} \Theta ^{2}
-2V_{i}\Theta\theta +V_{i}\theta^{2}-v_{i}\Theta^{2}+2v_{i}\Theta\theta-v_{i} \theta^{2}]  d\ell\\
=&-\partial_{i}(v_{i} \theta^{2} )^{\varepsilon}+2\partial_{i}(v_{i} \theta  )^{\varepsilon} \theta -\partial_{i}v_{i}^{\varepsilon} \theta^{2}+v_{i}\partial_{i}(  \theta^{2} )^{\varepsilon}-2v_{i}\partial_{i}  \theta ^{\varepsilon}\theta-v_{i} \theta^{2}\int_{\mathbb{T}^{3}} \partial_{l_{i}}\varphi_{\varepsilon}(\ell)   d\ell\\
=&\partial_{i}\B(v_{i}(  \theta^{2} )^{\varepsilon}-(v_{i} \theta^{2} )^{\varepsilon}\B)+2\partial_{i}(v_{i} \theta  )^{\varepsilon} \theta-2v_{i}\partial_{i}  \theta ^{\varepsilon}\theta,
\ea$$
which implies
\be\ba\label{key1}
&\int_{\mathbb{T}^{3}}\nabla\varphi_{\varepsilon}(\ell)\cdot\delta v(\ell)|\delta  \theta(\ell)|^{2}d\ell\\
=&\partial_{i}\B(v_{i}(  \theta^{2} )^{\varepsilon}-(v_{i} \theta^{2} )^{\varepsilon}\B)+2\partial_{i}(v_{i} \theta  )^{\varepsilon} \theta-2v_{i}\partial_{i}  \theta ^{\varepsilon}\theta.\ea\ee
In the light of condition
 \eqref{the1.1c}, we see that $\partial_{i}\B(v_{i}(  \theta^{2} )^{\varepsilon}-(v_{i} \theta^{2} )^{\varepsilon}\B)$ converges to $0$ in the sense of distributions on $(0,T)\times\mathbb{T}^{3}$ as $\varepsilon\rightarrow0$. Hence, we are in a position to finish the first part of the proof of Theorem \ref{the1.1}. It remains to establish \eqref{DRYaglom4/3law}.

Utilizing the polar coordinates and changing variables several times, we know that
\be\ba\label{key2}
D_{\varepsilon}(\theta,v)=&-\f14\int_{\mathbb{T}^{3}}
\nabla\varphi_{\varepsilon}(\ell)\cdot\delta v(\ell)|\delta  \theta(\ell)|^{2}d\ell\\
=&-\f14\int_{\mathbb{T}^{3}}
 \f{1}{\varepsilon^{4}}\varphi'(\f{|\ell|}{\varepsilon})\f{\ell}{|\ell|}\cdot\delta v(\ell)|\delta  \theta(\ell)|^{2}d\ell\\
 =&-\f14\int_{\mathbb{T}^{3}}
 \f{1}{\varepsilon }\varphi'( |\xi| )\f{\xi}{|\xi |}\cdot\delta v(\xi\varepsilon)|\delta  \theta(\xi\varepsilon)|^{2}d\xi\\
=&-\f14\int_{0}^{\infty}\int_{\partial B  } \f{r^{2}}{\varepsilon }\varphi'( |\zeta r| )\f{\zeta}{|\zeta|}\cdot[v(x+ \zeta r\varepsilon)-v(x)]
[\theta(x+\zeta r\varepsilon)-\theta(x)]^{2}d\sigma(\zeta)dr\\
=&-\f14\int_{0}^{\infty}\int_{\partial B } \f{r^{3}}{r\varepsilon }\varphi'( r) \zeta \cdot[v(x+ \zeta r\varepsilon)-v(x)]
[\theta(x+\zeta r\varepsilon)-\theta(x)]^{2}d\sigma(\zeta)dr\\
=&-\pi\int_{0}^{\infty}r^{3}\varphi'( r)dr\int_{\partial B }\f{\zeta \cdot[v(x+ \zeta r\varepsilon)-v(x)]
[\theta(x+\zeta r\varepsilon)-\theta(x)]^{2}\f{d\sigma(\zeta)}{4\pi}}{r\varepsilon }.
\ea\ee
In view of integration by parts, we arrive at
\be\ba\label{key3}
\int_{0}^{\infty}r^{3}\varphi'( r)dr=&-3\int_{0}^{\infty}r^{2}\varphi ( r)dr\\
=&-\f{3}{4\pi}\int_{\mathbb{R}^{3}} \varphi ( \ell)d\ell\\
=&-\f{3}{4\pi}.
\ea\ee
Consequently, we know that
$$\ba
 D(\theta,v)=&\lim_{\varepsilon\rightarrow0}D_{\varepsilon}(\theta,v)\\
 =&-\pi\int_{0}^{\infty}r^{3}\varphi'( r)dr\lim_{\varepsilon\rightarrow0}\int_{\partial B }\f{\zeta \cdot[v(x+ \zeta r\varepsilon)-v(x)]
[\theta(x+\zeta r\varepsilon)-\theta(x)]^{2}\f{d\sigma(\zeta)}{4\pi}}{r\varepsilon }\\
=&-\f{3}{4} S(\theta,v).
\ea$$
This implies the desired result.
\end{proof}

\subsection{$4/3$ law of energy and cross-helicity in the inviscid MHD flow
  }
In this subsection, we mainly focus on the proof of  Yaglom's 4/3 law of the energy and cross-helicity conservation in the inviscid MHD flow by an application of Els\"asser variables  $u ,h $ and the original variables $v,b$, respectively.
\begin{proof}[Proof of Theorem \ref{the1.2}]
First,  by mollifying the  equation  \eqref{mhdElsasser}, we have
 \be\label{3.1}\left\{\ba
&(u_{j}^{\varepsilon})_{t}+\partial_{i}(h_{i}u_{j})^{\varepsilon}+\partial_{j}\Pi^{\varepsilon} =0, \\
&(h_{j}^{\varepsilon})_{t}+\partial_{i}(u_{i}h_{j})^{\varepsilon}+\partial_{j}\Pi^{\varepsilon} =0, \\
&\Div u^{\varepsilon}=\Div h^{\varepsilon}=0.
 \ea\right.\ee
On the one hand, multiplying $\eqref{3.1}_{1}$ by $u_{j}$ and $\eqref{3.1}_{2}$ by $h_{j}$, respectively,  one has
 \be\label{3.2}\left\{\ba
&(u_{j}^{\varepsilon})_{t}u_{j}+\partial_{i}(h_{i}u_{j})^{\varepsilon}u_{j}+\partial_{j}\Pi^{\varepsilon} u_{j} =0, \\
&(h_{j}^{\varepsilon})_{t}h_{j}+\partial_{i}(u_{i}h_{j})^{\varepsilon}h_{j}+\partial_{j}\Pi^{\varepsilon} h_{j} =0.
 \ea\right.\ee
By the same token,  it follows from  from \eqref{mhdElsasser} that
 \be\label{3.3}\left\{\ba
&(u_{j})_{t}u^{\varepsilon}_{j}+\partial_{i}(h_{i}u_{j})u_{j}^{\varepsilon}+\partial_{j}\Pi u_{j}^{\varepsilon}  =0, \\
&(h_{j})_{t}h_{j}^{\varepsilon}+\partial_{i}(u_{i}h_{j})h_{j}^{\varepsilon}+\partial_{j}\Pi h_{j} ^{\varepsilon} =0.
 \ea\right.\ee
Then, in combination with  \eqref{3.2} and \eqref{3.3}, we discover that
  \be\label{3.3-1}\left\{\ba
&(u_{j}^{\varepsilon}u_{j})_{t}+\partial_{i}(h_{i}u_{j})^{\varepsilon}u_{j}+\partial_{i}(h_{i}u_{j})u_{j}^{\varepsilon}
+\partial_{j}\Pi^{\varepsilon} u_{j}+\partial_{j}\Pi u_{j}^{\varepsilon} =0, \\
&(h_{j}^{\varepsilon}h_{j})_{t}+\partial_{i}(u_{i}h_{j})^{\varepsilon}h_{j}+\partial_{i}(u_{i}h_{j})h_{j}^{\varepsilon}
+\partial_{j}\Pi^{\varepsilon} h_{j} +\partial_{j}\Pi h_{j}^{\varepsilon} =0.
 \ea\right.\ee
On the other hand, by an elementary calculation, we have the following facts that
$$\ba
&\partial_{i}(h_{i}u_{j})^{\varepsilon}u_{j}+\partial_{i}(h_{i}u_{j})u_{j}^{\varepsilon}
 =\partial_{i}(h_{i}u_{j}u_{j}^{\varepsilon})+\partial_{i}(h_{i}u_{j})^{\varepsilon}u_{j}-
(h_{i}u_{j}) \partial_{i}u_{j}^{\varepsilon},\\
&\partial_{i}(u_{i}h_{j})^{\varepsilon}h_{j}+\partial_{i}(u_{i}h_{j})h_{j}^{\varepsilon}=
\partial_{i}(u_{i}h_{j}h_{j}^{\varepsilon})+\partial_{i}(u_{i}h_{j})^{\varepsilon}h_{j}-(h_{i}h_{j})\partial_{i}h_{j}^{\varepsilon}.
\ea$$
Eventually, substituting the above facts into \eqref{3.3-1}, we have
 \be\label{3.3-2}\left\{\ba
&\f{(u_{j}^{\varepsilon}u_{j})_{t}}{2}+
\f{\partial_{i}(h_{i}u_{j}u_{j}^{\varepsilon})}{2}
 +\f{\partial_{j}\Pi^{\varepsilon} u_{j}+\partial_{j}\Pi u_{j}^{\varepsilon}}{2} =-\f12\B(\partial_{i}(h_{i}u_{j})^{\varepsilon}u_{j}-
(h_{i}u_{j}) \partial_{i}u_{j}^{\varepsilon}\B), \\
&\f{(h_{j}^{\varepsilon}h_{j})_{t}}{2}
+\f{\partial_{i}(u_{i}h_{j}h_{j}^{\varepsilon})}{2}
+\f{\partial_{j}\Pi^{\varepsilon} h_{j} +\partial_{j}\Pi h_{j}^{\varepsilon}}{2} =-\f12\B(\partial_{i}(u_{i}h_{j})^{\varepsilon}h_{j}
-(u_{i}h_{j})\partial_{i}h_{j}^{\varepsilon}\B).
 \ea\right.\ee
 Moreover, it follows from   \eqref{mhdElsasser} that
 $$
 -\partial_{j}\partial_{j}\Pi=\partial_{i}\partial_{j}(u_{j}h_{i}),
 $$
 which together with the classical Calder\'on-Zygmund Theorem  leads to $\Pi\in L^{\f32}(0,T;L^{\f32}(\mathbb{T}^{3}))$.

For the equation $\eqref{3.3-2}_1$, in the light of  $u,h \in L^{3}(0,T;L^{3}(\mathbb{T}^{3}))$  and $\Pi\in L^{\f32}(0,T;L^{\f32}(\mathbb{T}^{3}))$, we know that
 $$h_{i}u_{j}u_{j}^{\varepsilon} +\Pi u_{j}^{\varepsilon}+\Pi^{\varepsilon} u_{j} $$  tends to $$h(|u|^{2}) +2\Pi u $$ in the sense of distributions on $(0,T)\times\mathbb{T}^{3}$ as $\varepsilon\rightarrow0$. Hence, we know that the limit of
 $$-\B(\partial_{i}(h_{i}u_{j})^{\varepsilon}u_{j}-
(h_{i}u_{j}) \partial_{i}u_{j}^{\varepsilon} \B)$$
    is the same as $$(u_{j}^{\varepsilon}u_{j})_{t} +\partial_{i}(h_{i}u_{j}u_{j}^{\varepsilon})
+\partial_{j}\Pi^{\varepsilon} u_{j}+\partial_{j}\Pi u_{j}^{\varepsilon}
  $$ as $\varepsilon\rightarrow0.$

Then following  exactly along the same line in the proof of  \eqref{key1}, we deduce that
 $$\ba
\int_{\mathbb{T}^{3}}\nabla\varphi_{\varepsilon}(\ell)\cdot\delta h(\ell)|\delta  u(\ell)|^{2}d\ell
=&-\partial_{i}(h_{i} u_{j}^{2} )^{\varepsilon}+2\partial_{i}(h_{i} u_{j}  )^{\varepsilon} u_{j}  +h_{i}\partial_{i}(  u_{j}^{2} )^{\varepsilon}-2h_{i}\partial_{i}  u_{j} ^{\varepsilon}u_{j}\\
=& \partial_{i}\B(h_{i}(  u_{j}^{2} )^{\varepsilon}-(h_{i} u_{j}^{2} )^{\varepsilon}\B)+2\partial_{i}(h_{i} u_{j}  )^{\varepsilon} u_{j}   -2h_{i}\partial_{i}  u_{j} ^{\varepsilon}u_{j}.
\ea$$
Thanks to $h,u\in L^{3}(0,T;L^{3}(\mathbb{T}^{3}))$, by virtue of the standard properties of  mollification,  we know that $\partial_{i}\B(h_i(  u_{j}^{2} )^{\varepsilon}-(h_{i} u_{j}^{2} )^{\varepsilon}\B)\rightarrow0$ as $\varepsilon\rightarrow0$ in the sense of distributions.
 At this stage,  from the above argument, we see that $D_{\varepsilon}( u,h)$ converges towards  $$\partial_{t}(\f12|u|^{2})  +\text{div}\B(h(\f12|u| ^{2})+u\Pi\B).$$
Reasoning as above, we can prove
  $D_{\varepsilon}( h, u)$ tends to $$\partial_{t}(\f12|h|^{2})  +\text{div}\B(u(\f12|h| ^{2})+h\Pi\B).$$
Now we turn our attention to showing
\eqref{DRYaglom4/3lawmhd}.
Repeating the derivation of  \eqref{key2}, we find
$$\ba
D_{\varepsilon}(u,h)=&-\f14\int_{\mathbb{T}^{3}}
\nabla\varphi_{\varepsilon}(|\ell|)\cdot\delta u(\ell)|\delta h(\ell)|^{2}d\ell\\
 =&-\pi\int_{0}^{\infty}r^{3}\varphi'( r)dr\int_{\partial B }\f{\zeta \cdot\B(u(x+ \zeta r\varepsilon)-u(x)\B)
\B(h(x+\zeta r\varepsilon)-h(x)\B)^{2}\f{d\sigma(\zeta)}{4\pi}}{r\varepsilon },\\
\ea$$
and
$$\ba
D_{\varepsilon}(h,u)=&-\f14\int_{\mathbb{T}^{3}}
\nabla\varphi_{\varepsilon}(|\ell|)\cdot\delta h(\ell)|\delta u(\ell)|^{2}d\ell\\
 =&-\pi\int_{0}^{\infty}r^{3}\varphi'( r)dr\int_{\partial B }\f{\zeta \cdot\B(h(x+ \zeta r\varepsilon)-h(x)\B)
\B(u(x+\zeta r\varepsilon)-u(x)\B)^{2}\f{d\sigma(\zeta)}{4\pi}}{r\varepsilon }.
\ea$$
Combining this with \eqref{key3}, we arrive at
\eqref{DRYaglom4/3lawmhd}. The proof of this theorem is completed.
\end{proof}

\begin{proof}[Proof of Corollary \ref{coro1.3}]
With Theorem \ref{the1.2} in hand, we immediately get this corollary by a direct computation. Here, we omit the details.
\end{proof}
Next, we are devoted to the proof of Theorem \ref{the1.4} by an application of the original MHD equations.
\begin{proof}[Proof of Theorem \ref{the1.4}]
Since the proof of  Yaglom's 4/3 law of energy  in MHD flow is similar to the proof of this theorem, we mainly focus on the Yaglom's 4/3 law for cross helicity.
Mollifying the  \eqref{MHD}, we deduce that
$$\left\{\ba
&(v_{j}^{\varepsilon})_{t} +\partial_{i}(v_{i}  v_{j})^{\varepsilon}-\partial_{i}(b_{i}  b_{j})^{\varepsilon}+\partial_{j}\Pi ^{\varepsilon}=0, \\
&(b_{j}^{\varepsilon})_{t} +\partial_{i}(v_{i}  b_{j})^{\varepsilon}-\partial_{i}(v_{j}  b_{i})^{\varepsilon} =0, \\
&\Div v^{\varepsilon}=\Div b^{\varepsilon}=0,
 \ea\right.$$
from which it follows that
 \be\left\{\ba\label{MHD3.10}
&(v_{j}^{\varepsilon})_{t}b_{j} +\partial_{i}(v_{i}  v_{j})^{\varepsilon}b_{j}-\partial_{i}(b_{i}  b_{j})^{\varepsilon}b_{j}+ \partial_{j}\Pi ^{\varepsilon}b_{j}=0, \\
&(b_{j}^{\varepsilon})_{t} v_{j}+\partial_{i}(v_{i}  b_{j})^{\varepsilon}v_{j}-\partial_{i}(v_{j}  b_{i})^{\varepsilon}v_{j} =0.
 \ea\right.\ee
Likewise, we have
 \be\left\{\ba \label{MHD3.11}
&(v_{j})_{t}b_{j}^{\varepsilon} +\partial_{i}(v_{i}  v_{j})b_{j}^{\varepsilon}-\partial_{i}(b_{i}  b_{j})b_{j}^{\varepsilon}+ \partial_{j}\Pi b_{j}^{\varepsilon}=0, \\
&(b_{j})_{t} v_{j}^{\varepsilon}+\partial_{i}(v_{i}  b_{j})v_{j}^{\varepsilon}-\partial_{i}(v_{j}  b_{i})v_{j}^{\varepsilon} =0. \\
  \ea\right.\ee
According to \eqref{MHD3.10} and \eqref{MHD3.11}, we see that
\be\ba\label{2.11}
&(v_{j}^{\varepsilon}b_{j})_{t} +(v_{j}b_{j}^{\varepsilon})_{t}
+\partial_{i}(v_{i}  v_{j})^{\varepsilon}b_{j}+\partial_{i}(v_{i}  v_{j})b_{j}^{\varepsilon}
-\partial_{i}(b_{i}  b_{j})^{\varepsilon}b_{j}-\partial_{i}(b_{i}  b_{j})b_{j}^{\varepsilon}
\\&+ \partial_{j}\Pi ^{\varepsilon}b_{j}+ \partial_{j}\Pi b_{j}^{\varepsilon}
+\partial_{i}(v_{i}  b_{j})^{\varepsilon}v_{j}+\partial_{i}(v_{i}  b_{j})v_{j}^{\varepsilon}
-\partial_{i}(v_{j}  b_{i})^{\varepsilon}v_{j} -\partial_{i}(v_{j}  b_{i})v_{j}^{\varepsilon}=0.
\ea\ee
In the light of a calculation, one has
\be\ba\label{2.12}
&\partial_{i}(v_{i}  v_{j})^{\varepsilon}b_{j}+\partial_{i}(v_{i}  v_{j})b_{j}^{\varepsilon}
=\partial_{i}(v_{i}  v_{j}b_{j}^{\varepsilon})+\partial_{i}(v_{i}  v_{j})^{\varepsilon}b_{j}-
 v_{i}  v_{j}\partial_{i}b_{j}^{\varepsilon},
\\
&-\B(\partial_{i}(b_{i}  b_{j})^{\varepsilon}b_{j}+\partial_{i}(b_{i}  b_{j})b_{j}^{\varepsilon}\B)=-\B(
\partial_{i}(b_{i}  b_{j}b_{j}^{\varepsilon})+\partial_{i}(b_{i}  b_{j})^{\varepsilon}b_{j}-(b_{i}  b_{j})\partial_{i}b_{j}^{\varepsilon}\B),
\\
&\partial_{i}(v_{i}  b_{j})^{\varepsilon}v_{j}+\partial_{i}(v_{i}  b_{j})v_{j}^{\varepsilon}=
\partial_{i}(v_{i}  b_{j}v_{j}^{\varepsilon})+\partial_{i}(v_{i}  b_{j})^{\varepsilon}v_{j}
-(v_{i}  b_{j})\partial_{i}v_{j}^{\varepsilon},
\\
&-\B(\partial_{i}(v_{j}  b_{i})^{\varepsilon}v_{j} +\partial_{i}(v_{j}  b_{i})v_{j}^{\varepsilon}\B)
=-\B(\partial_{i}(v_{j}  b_{i}v_{j}^{\varepsilon})+\partial_{i}(v_{j}  b_{i})^{\varepsilon}v_{j} -
(v_{j}  b_{i})\partial_{i}v_{j}^{\varepsilon}\B).
\ea\ee
Then plugging  \eqref{2.12} into  \eqref{2.11}, we arrive at
\be\label{2.13}\ba
&\f{(v_{j}^{\varepsilon}b_{j})_{t} +(v_{j}b_{j}^{\varepsilon})_{t}}{2}
+\f{\partial_{i}(v_{i}  v_{j}b_{j}^{\varepsilon})+\partial_{i}(v_{i}  b_{j}v_{j}^{\varepsilon})-\partial_{i}(b_{i}  b_{j}b_{j}^{\varepsilon})-\partial_{i}(v_{j}  b_{i}^{\varepsilon}v_{j})}{2} +\f{\partial_{j}\Pi ^{\varepsilon}b_{j}+ \partial_{j}\Pi b_{j}^{\varepsilon}}{2}\\=&-\f{1}{2}\B(\partial_{i}(v_{i}  v_{j})^{\varepsilon}b_{j}-(v_{i}  v_{j})\partial_{i}b_{j}^{\varepsilon}\B)-\f12\B(\partial_{i}(v_{i}  b_{j})^{\varepsilon}v_{j}
-(v_{i}  b_{j})\partial_{i}v_{j}^{\varepsilon}\B)\\&
+\f12\B(\partial_{i}(b_{i}  b_{j})^{\varepsilon}b_{j}-(b_{i}  b_{j})\partial_{i}b_{j}^{\varepsilon}\B)
+\f12\B(\partial_{i}(v_{j}  b_{i})^{\varepsilon}v_{j} -
(v_{j}  b_{i})\partial_{i}v_{j}^{\varepsilon}\B).
\ea\ee
Next, applying the divergence operator to equation $\eqref{MHD}_{1}$, one has
$$
-\partial_{i}\partial_{i}\Pi=\partial_{i}\partial_{j}(v_{i}v_{j})
-\partial_{i}\partial_{j}(b_{i}b_{j}),
$$
which together with
$v, b\in L^{3}(0,T;L^{3}(\mathbb{T}^{3}))$ and  the
classical
Calder\'on-Zygmund
Theorem yields that    $\Pi\in L^{\f32}(0,T;L^{\f32}(\mathbb{T}^{3}))$. Moreover,
$v, b\in L^{3}(0,T;L^{3}(\mathbb{T}^{3}))$ and $\Pi\in L^{\f32}(0,T;L^{\f32}(\mathbb{T}^{3}))$ guarantee that the limit of the left hand side of
\eqref{2.13} is that
$$\partial_{t}(  v\cdot b)  +\text{div}\B(v(v\cdot b)-\f12(|b|^{2})b-\f12b(|v|^{2})+b\Pi\B)$$
in the sense of distributions as $\varepsilon\rightarrow0$.

Meanwhile, exactly as the derivation of \eqref{key1}, we discover that
\be\label{2.14}\ba
&\int_{\mathbb{T}^{3}}\nabla\varphi_{\varepsilon}(\ell)\cdot\delta v(\ell)|\delta  b(\ell)\delta  v(\ell)| d\ell\\
=&-\partial_{i}(v_{i}b_{j}v_{j})^{\varepsilon}+\partial_{i}(v_{i}b_{j})^{\varepsilon}v_{j}
+\partial_{i}(v_{i}v_{j})^{\varepsilon}b_{j}-\partial_{i}v_{i}^{\varepsilon}v_{j}b_{j}
+v_{i}\partial_{i}(b_{j}v_{j})^{\varepsilon}-v_{i}v_{j}\partial_{i}b_{j}^{\varepsilon}
-v_{i}b_{j}\partial_{i}v_{j}^{\varepsilon}\\
=&-\partial_{i}(v_{i}b_{j}v_{j})^{\varepsilon}+\partial_{i}(v_{i}b_{j})^{\varepsilon}v_{j}
+\partial_{i}(v_{i}v_{j})^{\varepsilon}b_{j}
+v_{i}\partial_{i}(b_{j}v_{j})^{\varepsilon}-v_{i}v_{j}\partial_{i}b_{j}^{\varepsilon}
-v_{i}b_{j}\partial_{i}v_{j}^{\varepsilon} \\
=&\partial_{i}\B(v_{i}(b_{j}v_{j})^{\varepsilon}-(v_{i}b_{j}v_{j})^{\varepsilon}\B)
+\partial_{i}(v_{i}v_{j})^{\varepsilon}b_{j}
 -v_{i}v_{j}\partial_{i}b_{j}^{\varepsilon}
 +\partial_{i}(v_{i}b_{j})^{\varepsilon}v_{j}
-v_{i}b_{j}\partial_{i}v_{j}^{\varepsilon},\ea\ee
where the divergence-free conditions $\Div v=0$ and $\Div b=0$ are used.
Following a similar argument, we also find
\be\label{2.15}\ba
\int_{\mathbb{T}^{3}}\nabla\varphi_{\varepsilon}(\ell)\cdot\delta b(\ell)|\delta  b(\ell)|^{2}d\ell
=& \partial_{i}\B(b_{i}(  b_{j}^{2} )^{\varepsilon}-(b_{i} b_{j}^{2} )^{\varepsilon}\B)+2\partial_{i}(b_{i} b_{j}  )^{\varepsilon} b_{j}   -2b_{i}\partial_{i}  b_{j} ^{\varepsilon}b_{j},\ea\ee
and
\be\label{2.16}\ba
\int_{\mathbb{T}^{3}}\nabla\varphi_{\varepsilon}(\ell)\cdot\delta b(\ell)|\delta  v(\ell)|^{2}d\ell
 =& \partial_{i}\B(b_{i}( v_{j}^{2} )^{\varepsilon}-(b_{i} v_{j}^{2} )^{\varepsilon}\B)+2\partial_{i}(b_{i} v_{j}  )^{\varepsilon} v_{j}   -2b_{i}\partial_{i}  v_{j} ^{\varepsilon} v_{j}.
 \ea \ee
 We conclude by $v, b\in L^{3}(0,T;L^{3}(\mathbb{T}^{3}))$ that all the first terms on the right hand side of \eqref{2.14}-\eqref{2.16}  tend to zero in the sense of distributions as $\varepsilon\rightarrow0$.
 Collecting the above estimates together and the definition of $ D_{CH}( v,b,\varepsilon)$, we infer that the limit of $ D_{CH}( v,b,\varepsilon)$ is the same as that of the right hand side of \eqref{2.13}
in the sense of distributions as $\varepsilon\rightarrow0$. Hence, we finish the proof of the first part in this theorem.

We turn our attention to the rest of this proof. In the light of a  slight variant of the proof of \eqref{key2}, we derive that
  $$\ba
 & D_{CH}( v,b,\varepsilon)\\
 =&-\f12\int_{\mathbb{T}^{3}}\nabla\varphi_{\varepsilon}(\ell)\cdot\delta v(\ell) |\delta  v(\ell)\delta  b(\ell)| d\ell+\f14\int_{\mathbb{T}^{3}}\nabla\varphi_{\varepsilon}(\ell)\cdot\delta b(\ell)(| \delta  v(\ell)|^{2}|+| \delta  b(\ell)|^{2}|) d\ell\\
  =&-2\pi\int_{0}^{\infty}r^{3}\varphi'( r)dr\int_{\partial B }\f{\zeta \cdot\B(v(x+ \zeta r\varepsilon)-v(x)\B)
\B((v(x+\zeta r\varepsilon)-v(x))(b(x+\zeta r\varepsilon)-b(x))\B)\f{d\sigma(\zeta)}{4\pi}}{r\varepsilon }\\&+\pi\int_{0}^{\infty}r^{3}\varphi'( r)dr\int_{\partial B }\f{\zeta \cdot\B(b(x+ \zeta r\varepsilon)-b(x)\B)
\B(v(x+\zeta r\varepsilon)-v(x)\B)^{2}\f{d\sigma(\zeta)}{4\pi}}{r\varepsilon }\\&+\pi\int_{0}^{\infty}r^{3}\varphi'( r)dr\int_{\partial B }\f{\zeta \cdot\B(b(x+ \zeta r\varepsilon)-b(x)\B)
\B(b(x+\zeta r\varepsilon)-b(x)\B)^{2}\f{d\sigma(\zeta)}{4\pi}}{r\varepsilon }.
\ea$$
Combining this and \eqref{key3}, we get
$$
D_{CH}( v,b )=-\f32S_{4}( v,v,b )+\f34S_{5}( b,v,v )+\f34S_{6}( b,b,b ).$$
This enables us to complete the proof.
\end{proof}
\subsection{$4/3$ law of helicity in the inviscid Euler fluid
  }

In this final subsection, we are concerned with the proof of Yaglom's 4/3 law for the helicity of ideal incompressible Euler equations.
\begin{proof}[Proof of Theorem \ref{the1.5}]
Mollifying the equation \eqref{euler}, we have
\begin{equation}\label{euler2}\left\{\begin{aligned}
&v^{\varepsilon}_{t} +(v\cdot\nabla v)^{\varepsilon}+\nabla \pi^{\varepsilon} = 0,\\
&\omega^{\varepsilon}_{t}+(v\cdot\nabla\omega)^{\varepsilon}-(\omega\cdot\nabla v)^{\varepsilon}=0 ,\\
&\Div v^{\varepsilon}= \text{div\,}\omega^{\varepsilon} =0.
\end{aligned}\right.\end{equation}
Then we write
\be\ba\label{2.18}
(v_{j}\omega_{j}^{\varepsilon})_{t}+(v_{j}^{\varepsilon}\omega_{j})_{t}
+I+II+III-IV=0,\ea\ee
where
$$\ba
&I=\partial_{i}(v_{i}v_{j})^{\varepsilon}\omega_{j}+\partial_{i}(v_{i}v_{j})\omega_{j}^{\varepsilon},\\
&II=\partial_{j}\pi^{\varepsilon}\omega_{j}
+\partial_{j}\pi\omega_{j}^{\varepsilon},\\
&III=\partial_{i}(v_{i}\omega_{j})^{\varepsilon}v_{j}
+\partial_{i}(v_{i}\omega_{j})v_{j}^{\varepsilon},\\
&IV=\partial_{i}(\omega_{i}v_{j}^{\varepsilon})v_{j}+\partial_{i}(\omega_{i}v_{j})v_{j}^{\varepsilon}.
\ea$$
By some calculations,  one has
$$\ba
&\partial_{i}(v_{i}v_{j})^{\varepsilon}\omega_{j}+
\partial_{i}(v_{i}v_{j})\omega_{j}^{\varepsilon}
=\partial_{i}(v_{i}v_{j}\omega_{j}^{\varepsilon})+
\partial_{i}(v_{i}v_{j})^{\varepsilon}\omega_{j}-
(v_{i}v_{j})\partial_{i}\omega_{j}^{\varepsilon},\\
 &\partial_{i}(v_{i}\omega_{j})^{\varepsilon}v_{j}
+\partial_{i}(v_{i}\omega_{j})v_{j}^{\varepsilon}=
\partial_{i}(v_{i}\omega_{j}v_{j}^{\varepsilon})+
\partial_{i}(v_{i}\omega_{j})^{\varepsilon}v_{j}
-(v_{i}\omega_{j})\partial_{i}v_{j}^{\varepsilon},\\
&\partial_{i}(\omega_{i}v_{j}^{\varepsilon})v_{j}+
\partial_{i}(\omega_{i}v_{j})v_{j}^{\varepsilon}
=\partial_{i}(\omega_{i}v_{j}v_{j}^{\varepsilon})+
\partial_{i}(\omega_{i}v_{j}^{\varepsilon})v_{j}
-(\omega_{i}v_{j})\partial_{i}v_{j}^{\varepsilon}.
\ea$$
Then substituting the above equations into \eqref{2.18}, we notice that
\be\ba\label{2.19}
&\f{(v_{j}\omega_{j}^{\varepsilon})_{t}
+(v_{j}^{\varepsilon}\omega_{j})_{t}}{2}
+\f{\partial_{i}(v_{i}v_{j}\omega_{j}^{\varepsilon})+
\partial_{i}(v_{i}\omega_{j}v_{j}^{\varepsilon})
-\partial_{i}(\omega_{i}v_{j}v_{j}^{\varepsilon})}{2}
+\f{\partial_{j}\pi^{\varepsilon}\omega_{j}
+\partial_{j}\pi\omega_{j}^{\varepsilon}}{2}\\
=&-\f{1}{2}\B(\partial_{i}(v_{i}v_{j})^{\varepsilon}\omega_{j}-
(v_{i}v_{j})\partial_{i}\omega_{j}^{\varepsilon}\B)-
\f{1}{2}\B(\partial_{i}(v_{i}\omega_{j})^{\varepsilon}v_{j}
-(v_{i}\omega_{j})\partial_{i}v_{j}^{\varepsilon}\B)+
\f{1}{2}\B(\partial_{i}(\omega_{i}v_{j}^{\varepsilon})v_{j}
-(\omega_{i}v_{j})\partial_{i}v_{j}^{\varepsilon}\B).
\ea\ee
Meanwhile, it follows from $\eqref{euler}_{1}$ that $$-\partial_{i}\partial_{i}\pi=\partial_{i}\partial_{j}(v_{i}v_{j}),$$
which turns out that
$$
\|\pi\|_{L^{\f{p}{2}}(0,T;L^{\f{q}{2}}(\mathbb{T}^{3}))} \leq C
\|v\|^{2}_{L^{p}(0,T;L^{q}(\mathbb{T}^{3}))}, 2<q<\infty.$$
Consequently, we deduce from $v \in L^{q}(0,T;L^{q}(\mathbb{T}^{3}))$ and $\pi\in L^{\f{p}{2}}(0,T;L^{\f{q}{2}}(\mathbb{T}^{3}))$ with $3\leq p,q<\infty$ that
the limits of the LHS of
\eqref{2.19} is that
$$\partial_{t}( v\cdot\omega)  +\text{div}\B(v(\omega\cdot v)-\f12\omega(v^{2})+\omega\pi\B)$$
in the sense of distribution as $\varepsilon\rightarrow0$. To pass to the limits of RHS of
\eqref{2.19}, a slightly modified proof of \eqref{key1} yields that
$$\ba
&\int_{\mathbb{T}^{3}}\nabla\varphi_{\varepsilon}(\ell)\cdot\delta v(\ell)|\delta  \omega(\ell)\delta  v(\ell)| d\ell\\
 =&\partial_{i}\B(v_{i}(\omega_{j}v_{j})^{\varepsilon}-(v_{i}\omega_{j}v_{j})^{\varepsilon}\B)
+\partial_{i}(v_{i}v_{j})^{\varepsilon}\omega_{j}
 -v_{i}v_{j}\partial_{i}\omega_{j}^{\varepsilon}
 +\partial_{i}(v_{i}\omega_{j})^{\varepsilon}v_{j}
-v_{i}\omega_{j}\partial_{i}v_{j}^{\varepsilon}
\ea$$
and
$$\ba
\int_{\mathbb{T}^{3}}\nabla\varphi_{\varepsilon}(\ell)\cdot\delta \omega(\ell)|\delta  v(\ell)|^{2}d\ell
=&-\partial_{i}(\omega_{i} v_{j}^{2} )^{\varepsilon}+2\partial_{i}(\omega_{i} v_{j}  )^{\varepsilon} v_{j}  +\omega_{i}\partial_{i}(  v_{j}^{2} )^{\varepsilon}-2\omega_{i}\partial_{i}  v_{j} ^{\varepsilon}v_{j}\\
=& \partial_{i}\B(\omega_{i}(  v_{j}^{2} )^{\varepsilon}-(\omega_{i} v_{j}^{2} )^{\varepsilon}\B)+2\partial_{i}(\omega_{i} v_{j}  )^{\varepsilon} v_{j}   -2\omega_{i}\partial_{i}  v_{j} ^{\varepsilon}v_{j}.\ea$$
Thanks to \eqref{the1.5c}, we see  that both $\partial_{i}\B(v_{i}(\omega_{j}v_{j})^{\varepsilon}-(v_{i}\omega_{j}v_{j})^{\varepsilon}\B)
$  and $\partial_{i}\B(\omega_{i}(  v_{j}^{2} )^{\varepsilon}-(\omega_{i} v_{j}^{2} )^{\varepsilon}\B)$ tend to zero in the sense of distribution as $\varepsilon\rightarrow0$. With this in hand, we conclude that the limits of RHS of \eqref{2.19} is the same as that of  $D_{\varepsilon}( v,\omega)$. This shows the first part of this theorem.

To get \eqref{helicityYaglom4/3law}, repeating the derivation of \eqref{key2}, we infer that
$$\ba
&D_{\varepsilon}(v,\omega)\\=&-\f12\int_{\mathbb{T}^{3}}\nabla\varphi_{\varepsilon}(\ell)\cdot\delta v(\ell)|\delta  \omega(\ell)\cdot\delta  v(\ell)| d\ell+\f14\int_{\mathbb{T}^{3}}\nabla\varphi_{\varepsilon}(\ell)\cdot\delta \omega(\ell)|\delta  v(\ell)|^{2}d\ell\\
=&-2\pi\int_{0}^{\infty}r^{3}\varphi'( r)dr\int_{\partial B }\f{\zeta \cdot\B(v(x+ \zeta r\varepsilon)-v(x)\B)
\B((v(x+\zeta r\varepsilon)-v(x))(\omega(x+\zeta r\varepsilon)-\omega(x))\B)\f{d\sigma(\zeta)}{4\pi}}{r\varepsilon }\\&+\pi\int_{0}^{\infty}r^{3}\varphi'( r)dr \int_{\partial B }\f{\zeta \cdot\B(\omega(x+ \zeta r\varepsilon)-\omega(x)\B)
\B(v(x+\zeta r\varepsilon)-v(x)\B)^{2}\f{d\sigma(\zeta)}{4\pi}}{r\varepsilon }.
\ea$$
As an application of \eqref{key3}, we find
$$
D(v,\omega)=-\f32 S_{7}(v,v,\omega)+\f34 S_{8}(\omega,v,v),
$$
which implies \eqref{helicityYaglom4/3law}.
This concludes the proof of this theorem.
\end{proof}
As an application of  the dissipation term  $D_{\varepsilon}( v,\omega)$ in Theorem  \ref{the1.5}, we prove Corollary \ref{coro1.6}.
\begin{proof}[Proof of Corollary \ref{coro1.6}]
In view of the H\"older inequality, we obtain
$$
\int_{\mathbb{T}^{3}} |D_{\varepsilon}( v,\omega)|dx\leq \int_{\mathbb{T}^{3}}|\nabla\varphi_{\varepsilon}(\ell)|d\ell\B(\int_{\mathbb{T}^{3}}|\delta v(\ell)|^{\f{9}{2}}dx\B)^{\f49}\B(\int_{\mathbb{T}^{3}} |\delta \omega(\ell)|^{\f{9}{5}} dx\B)^{\f{5}{9}}.
$$
By virtue of \eqref{1.23}, we get
 $$\ba
\int_{\mathbb{T}^{3}} |D_{\varepsilon}( v,\omega)|dx\leq & \int_{\mathbb{T}^{3}}|\nabla\varphi_{\varepsilon}(\ell)|d\ell\B(\int_{\mathbb{T}^{3}}|\delta v(\ell)|^{\f{9}{2}}dx\B)^{\f49}\B(\int_{\mathbb{T}^{3}} |\delta \omega(\ell)|^{\f{9}{5}} dx\B)^{\f{5}{9}}\\
\leq&   \int_{\mathbb{T}^{3}}|\nabla\varphi_{\varepsilon}(\ell)| C(t)^{\f{2}{r_{1}}+\f{1}{r_{2}}}|\ell|^{2\alpha+\beta}\sigma(\ell)d\ell.
\ea$$
Performing a time integration and changing variable, we end up with
$$\ba
 \int_{0}^{T}\int_{\mathbb{T}^{3}} |D_{\varepsilon}( v,\omega)|dxdt \leq&
\int_{0}^{T}C(t)^{\f{2}{r_{1}}+\f{1}{r_{2}}}dt\int_{\mathbb{T}^{3}}
|\nabla\varphi_{\varepsilon}(\ell)|
|\ell|^{2\alpha+\beta}\sigma (\ell)d\ell\\
\leq&C \varepsilon^{2\alpha+\beta-1}\int_{|\xi|<1}|\nabla\varphi (\xi)
||\xi|^{2\alpha+\beta}\sigma (\varepsilon\xi)d\xi.
\ea$$
This means the desired result.
\end{proof}
\section{New Yaglom's laws   for  the Oldroyd-B model  and subgrid scale $\alpha$-models of turbulence  }
\subsection{ $4/3$ law   in Oldroyd-B model  }
  \begin{proof}[Proof of Theorem \ref{the1.6}] We just outline the proof here.
  The pressure equation in Oldroyd-B model \eqref{Oldroyd-B3} is   determined   by
  $$
-\partial_{j}\partial_{j}\Pi=\partial_{j}\partial_{i}(v_{i}v_{j})
-\partial_{j}\partial_{i}(\tau_{i,j}),
$$
  which means that
$$\ba
\|\pi\|_{L^{\f{3}{2}}(0,T;L^{\f{3}{2}}(\mathbb{T}^{3}))} \leq& C
\|v\|^{2}_{L^{3}(0,T;L^{3}(\mathbb{T}^{3}))}+C
\|\tau\|_{L^{\f{3}{2}}(0,T;L^{\f{3}{2}}(\mathbb{T}^{3}))}\\
\leq& C
\|v\|^{2}_{L^{3}(0,T;L^{3}(\mathbb{T}^{3}))}+C
\|\tau\|_{L^{3}(0,T;L^{3}(\mathbb{T}^{3}))}.\ea$$
  As the above deduction, from the Oldroyd-B model \eqref{Oldroyd-B3},
  \be\ba\label{03.1}
 & \partial_{t}(v_{j}v_{j}^{\varepsilon})+\partial_{i}(v_{i}v_{j})v_{j}^{\varepsilon}
  +\partial_{i}(v_{i}v_{j})^{\varepsilon}v_{j}+\partial_{j}\pi v_{j}^{\varepsilon}+\partial_{j}\pi ^{\varepsilon}  v_{j}=\partial_{i}\tau_{ij}v_{j}^{\varepsilon}+\partial_{i}\tau_{ij}^{\varepsilon}v_{j}\\
& \partial_{t}(\tau_{ij}\tau_{ij}^{\varepsilon})+
\partial_{k}(v_{k}\tau_{ij})^{\varepsilon}\tau_{ij}
+\partial_{k}(v_{k}\tau_{ij})\tau_{ij}^{\varepsilon}+2\tau_{ij}\tau_{ij}^{\varepsilon}
=\f{\partial_{i}v_{j}^{\varepsilon}+ \partial_{j}v_{i}^{\varepsilon}}{2}\tau_{ij}+ \tau_{ij}^{\varepsilon}\f{\partial_{i}v_{j} + \partial_{j}v_{i}}{2}.
\ea\ee
By a straightforward computation,  we notice that
$$  \ba &
\partial_{i}\tau_{ij}v_{j}^{\varepsilon}=
\partial_{i}[\tau_{ij}v_{j}^{\varepsilon}]-\tau_{ij}\partial_{i}v_{j}^{\varepsilon},\ \partial_{i}\tau_{ij}^{\varepsilon}v_{j}=
\partial_{i}[\tau_{ij}^{\varepsilon}v_{j}]-\tau_{ij}^{\varepsilon}\partial_{i}v_{j}
\ea$$
  In view of  the symmetric of tensor $\tau$, we discover that
  $$\ba
  \tau_{ij}\partial_{i}v_{j}^{\varepsilon}=  \tau_{ij}\f{\partial_{i}v_{j}^{\varepsilon}+ \partial_{j}v_{i}^{\varepsilon}}{2},
 \tau_{ij}^{\varepsilon}\partial_{i}v_{j}=  \tau_{ij}^{\varepsilon}\f{\partial_{i}v_{j} + \partial_{j}v_{i}}{2},  \ea $$
 which from follows that
\be  \ba\label{03.2}
& \partial_{i}\tau_{ij}v_{j}^{\varepsilon}=
\partial_{i}[\tau_{ij}v_{j}^{\varepsilon}]- \tau_{ij}\f{\partial_{i}v_{j}^{\varepsilon}+ \partial_{j}v_{i}^{\varepsilon}}{2}\\
&\partial_{i}\tau_{ij}^{\varepsilon}v_{j}=
\partial_{i}[\tau_{ij}^{\varepsilon}v_{j}]- \tau_{ij}^{\varepsilon}\f{\partial_{i}v_{j} + \partial_{j}v_{i}}{2}.
\ea\ee
Moreover, it is clear that
\be\ba\label{03.3}
&\partial_{i}(v_{i}v_{j})v_{j}^{\varepsilon}
  +\partial_{i}(v_{i}v_{j})^{\varepsilon}v_{j}=\partial_{i}(v_{i}v_{j}v_{j}^{\varepsilon})
  +\partial_{i}(v_{i}v_{j})^{\varepsilon}v_{j}
  -(v_{i}v_{j})\partial_{i}v_{j}^{\varepsilon},
 \\
&\partial_{k}(v_{k}\tau_{ij})\tau_{ij}^{\varepsilon}+ \partial_{k}(v_{k}\tau_{ij})^{\varepsilon}\tau_{ij}
=\partial_{k}(v_{k}\tau_{ij}\tau_{ij}^{\varepsilon})+ \partial_{k}(v_{k}\tau_{ij})^{\varepsilon}\tau_{ij}
 -(v_{k}\tau_{ij})\partial_{k}\tau_{ij}^{\varepsilon}.
 \ea\ee
Inserting \eqref{03.2} and \eqref{03.3} into \eqref{03.1}, we see that
\be\label{3.4}\ba
&\f{\partial_{t}(v_{j}v_{j}^{\varepsilon})+\partial_{t}(\tau_{ij}\tau_{ij}^{\varepsilon})}{2}
+\f{\partial_{i}(v_{i}v_{j}v_{j}^{\varepsilon})+\partial_{k}(v_{k}\tau_{ij}\tau_{ij}^{\varepsilon})}{2}
+\f{\partial_{j}(\pi v_{j}^{\varepsilon}+\pi ^{\varepsilon}  v_{j})}{2}\\&-\partial_{i}[\tau_{ij}v_{j}^{\varepsilon}
+\tau_{ij}^{\varepsilon}v_{j}]+\tau_{ij}\tau_{ij}^{\varepsilon}\\=&
 -\f{1}{2} \B[\partial_{i}(v_{i}v_{j})^{\varepsilon}v_{j}-(v_{i}v_{j})\partial_{i}v_{j}^{\varepsilon}\B]
 -\f{1}{2} \B[ \partial_{k}(v_{k}\tau_{ij})^{\varepsilon}\tau_{ij}
 -(v_{k}\tau_{ij})\partial_{k}\tau_{ij}^{\varepsilon}\B]
 \ea\ee
With a similar technique as the one
used for \eqref{key1}, we obtain
\be\ba\label{3.5}\int_{\mathbb{T}^{3}}\nabla\varphi_{\varepsilon}(\ell)\cdot\delta v(\ell)|\delta  v(\ell)|^{2}d\ell
 =& \partial_{i}\B(v_{i}( v_{j}^{2} )^{\varepsilon}-(v_{i} v_{j}^{2} )^{\varepsilon}\B)+2\partial_{i}(v_{i} v_{j}  )^{\varepsilon} v_{j}   -2v_{i}\partial_{i}  v_{j} ^{\varepsilon} v_{j},\ea\ee
and
 \be\ba\label{3.6}
\int_{\mathbb{T}^{3}}\nabla\varphi_{\varepsilon}(\ell)\cdot\delta v(\ell)|\delta  \tau_{ij}(\ell)|^{2}d\ell
=\partial_{k}\B(v_{k}( \tau_{ij}^{2} )^{\varepsilon}-(v_{i} \tau_{ij}^{2} )^{\varepsilon}\B)+2\partial_{k}(v_{k} \tau_{ij}  )^{\varepsilon} \theta-2v_{k}\partial_{k}  \tau_{ij}^{\varepsilon}\tau_{ij}.
\ea\ee
 Roughly speaking, $v,\tau \in   L^{3}(0,T;L^{3}(\mathbb{T}^{3}))$ and $\pi\in L^{\f{3}{2}}(0,T;L^{\f{3}{2}}(\mathbb{T}^{3}))$ guarantee the the limit of  $D_{\varepsilon}(v,\tau)$ is the same as that of left hand side of \eqref{3.4} as $\varepsilon\rightarrow0.$ According to the derivation of \eqref{key2}, \eqref{key3} and the definition $S(v,v,v), S(v,\tau,\tau)$, we conclude the proof of this theorem.
 \end{proof}
\subsection{ $4/3$ law for subgrid scale $\alpha$-models of turbulence
  }
  \begin{proof}[Proof of Theorem \ref{the1.7}]
  According to \cite[P4, Theorem 2.4]{[BT]}, we remark that
  $$D_{1}(u,v)=\lim\limits_{\varepsilon\rightarrow0}D_{\varepsilon}(u,v)=\lim\limits_{\varepsilon\rightarrow0}-\f14\int_{\mathbb{T}^{3}}\nabla\varphi_{\varepsilon}(\ell)\cdot\delta u(\ell)|\delta  v(\ell)|^{2}d\ell$$
Following the same path of   \eqref{key2} \eqref{key3} and the definition of $S(u,v,v)$, we get
 $$ \ba D(u,v)=&\lim_{\varepsilon\rightarrow0}D_{\varepsilon}(u,v)\\
 =&-\pi\int_{0}^{\infty}r^{3}\varphi'( r)dr\lim_{\varepsilon\rightarrow0}\int_{\partial B }\f{\zeta \cdot[u(x+ \zeta r\varepsilon)-u(x)]
[v(x+\zeta r\varepsilon)-v(x)]^{2}\f{d\sigma(\zeta)}{4\pi}}{r\varepsilon }\\
=&-\f{3}{4} S(u,v,v).\ea$$
This finishes the proof of the first part of this theorem. Likewise,  the rest proofs of this theorem are a combination of the corresponding results in \cite{[BT]} and the definition of $S_{i}(\cdot,\cdot,\cdot)$ given here. We omit the details here.
  \end{proof}
\begin{proof}[Proof of Theorem \ref{the1.8}]
From  $v\in L^{\infty}(0,T;L^{2}(\mathbb{T}^{3}))$, we observe that $u\in L^{\infty}(0,T;H^{2}(\mathbb{T}^{3}))$. Hence, we get  $u\in L^{3}(0,T;L^{3}(\mathbb{T}^{3}))$.
 Arguing as above, we deduce from \eqref{LMHD} that
\be\ba\label{3.7}
&(v_{j}^{\varepsilon}H_{j})_{t} +(v_{j}H_{j}^{\varepsilon})_{t}
+\partial_{i}(u_{i}  v_{j})^{\varepsilon}H_{j}+\partial_{i}(v_{i}  v_{j})H_{j}^{\varepsilon}
-\partial_{i}(H_{i}  H_{j})^{\varepsilon}H_{j}-\partial_{i}(H_{i}  H_{j})H_{j}^{\varepsilon}
\\&+ \partial_{j}\Pi ^{\varepsilon}H_{j}+ \partial_{j}\Pi H_{j}^{\varepsilon}
+\partial_{i}(u_{i}  H_{j})^{\varepsilon}v_{j}+\partial_{i}(u_{i}  H_{j})v_{j}^{\varepsilon}
-\partial_{i}(v_{j}  H_{i})^{\varepsilon}v_{j} -\partial_{i}(v_{j}  H_{i})v_{j}^{\varepsilon}=0.
\ea\ee
It is easy to check that
\be\ba\label{3.8}
&\partial_{i}(u_{i}  v_{j})^{\varepsilon}b_{j}+\partial_{i}(u_{i}  v_{j})H_{j}^{\varepsilon}
=\partial_{i}(u_{i}  v_{j}H_{j}^{\varepsilon})+\partial_{i}(u_{i}  v_{j})^{\varepsilon}H_{j}-
 u_{i}  v_{j}\partial_{i}H_{j}^{\varepsilon},
\\
&-\B(\partial_{i}(H_{i}  H_{j})^{\varepsilon}b_{j}+\partial_{i}(b_{i}  H_{j})H_{j}^{\varepsilon}\B)=-\B(
\partial_{i}(H_{i}  H_{j}H_{j}^{\varepsilon})+\partial_{i}(H_{i}  H_{j})^{\varepsilon}H_{j}-(H_{i}  H_{j})\partial_{i}H_{j}^{\varepsilon}\B),
\\
&\partial_{i}(u_{i}  H_{j})^{\varepsilon}v_{j}+\partial_{i}(u_{i}  H_{j})v_{j}^{\varepsilon}=
\partial_{i}(u_{i}  H_{j}v_{j}^{\varepsilon})+\partial_{i}(u_{i}  H_{j})^{\varepsilon}v_{j}
-(u_{i}  H_{j})\partial_{i}v_{j}^{\varepsilon},
\\
&-\B(\partial_{i}(v_{j}  H_{i})^{\varepsilon}v_{j} +\partial_{i}(v_{j}  H_{i})v_{j}^{\varepsilon}\B)
=-\B(\partial_{i}(v_{j}  H_{i}v_{j}^{\varepsilon})+\partial_{i}(v_{j}  H_{i})^{\varepsilon}v_{j} -
(v_{j}  H_{i})\partial_{i}v_{j}^{\varepsilon}\B).
\ea\ee
Inserting \eqref{3.7} into \eqref{3.8}, we gather that
\be\label{3.9}\ba
&\f{(v_{j}^{\varepsilon}H_{j})_{t} +(v_{j}H_{j}^{\varepsilon})_{t}}{2}
+\f{\partial_{i}(u_{i}  v_{j}H_{j}^{\varepsilon})+\partial_{i}(u_{i}  H_{j}v_{j}^{\varepsilon})-\partial_{i}(H_{i}  H_{j}H_{j}^{\varepsilon})-\partial_{i}(v_{j}  H_{i}^{\varepsilon}v_{j})}{2} +\f{\partial_{j}\Pi ^{\varepsilon}H_{j}+ \partial_{j}\Pi H_{j}^{\varepsilon}}{2}\\=&-\f{1}{2}\B(\partial_{i}(u_{i}  v_{j})^{\varepsilon}H_{j}-(u_{i}  v_{j})\partial_{i}H_{j}^{\varepsilon}\B)-\f12\B(\partial_{i}(u_{i}  H_{j})^{\varepsilon}v_{j}
-(u_{i}  H_{j})\partial_{i}v_{j}^{\varepsilon}\B)\\&
+\f12\B(\partial_{i}(H_{i}  H_{j})^{\varepsilon}H_{j}-(H_{i}  H_{j})\partial_{i}H_{j}^{\varepsilon}\B)
+\f12\B(\partial_{i}(v_{j}  H_{i})^{\varepsilon}v_{j} -
(v_{j}  H_{i})\partial_{i}v_{j}^{\varepsilon}\B).
\ea\ee
The pressure equation in $\eqref{LMHD}$ reads
$$
-\partial_{i}\partial_{i}\Pi=\partial_{i}\partial_{j}(u_{i}v_{j})
-\partial_{i}\partial_{j}(H_{i}H_{j}).
$$
Therefore,
$u,v, H\in L^{3}(0,T;L^{3}(\mathbb{T}^{3}))$ and  the
classical
elliptic estimate leads to that    $\Pi\in L^{\f32}(0,T;L^{\f32}(\mathbb{T}^{3}))$.

Meanwhile, exactly as the derivation of \eqref{key1}, we notice that
\be\label{3.10}\ba
&\int_{\mathbb{T}^{3}}\nabla\varphi_{\varepsilon}(\ell)\cdot\delta v(\ell)|\delta  H(\ell)\delta  v(\ell)| d\ell\\
=&-\partial_{i}(u_{i}H_{j}v_{j})^{\varepsilon}+\partial_{i}(u_{i}H_{j})^{\varepsilon}v_{j}
+\partial_{i}(u_{i}v_{j})^{\varepsilon}b_{j}-\partial_{i}u_{i}^{\varepsilon}v_{j}b_{j}
+v_{i}\partial_{i}(H_{j}v_{j})^{\varepsilon}-u_{i}v_{j}\partial_{i}H_{j}^{\varepsilon}
-u_{i}H_{j}\partial_{i}v_{j}^{\varepsilon}\\
=&-\partial_{i}(u_{i}H_{j}v_{j})^{\varepsilon}+\partial_{i}(u_{i}H_{j})^{\varepsilon}v_{j}
+\partial_{i}(u_{i}v_{j})^{\varepsilon}H_{j}
+u_{i}\partial_{i}(H_{j}v_{j})^{\varepsilon}-u_{i}v_{j}\partial_{i}H_{j}^{\varepsilon}
-u_{i}H_{j}\partial_{i}v_{j}^{\varepsilon} \\
=&\partial_{i}\B(u_{i}(H_{j}v_{j})^{\varepsilon}-(u_{i}H_{j}v_{j})^{\varepsilon}\B)
+\partial_{i}(u_{i}v_{j})^{\varepsilon}b_{j}
 -u_{i}v_{j}\partial_{i}H_{j}^{\varepsilon}
 +\partial_{i}(u_{i}H_{j})^{\varepsilon}v_{j}
-u_{i}H_{j}\partial_{i}v_{j}^{\varepsilon},\ea\ee
where the divergence-free conditions $\Div v=0$ and $\Div H=0$ are used.
Following a similar argument, we also find
\be\label{3.11}\ba
\int_{\mathbb{T}^{3}}\nabla\varphi_{\varepsilon}(\ell)\cdot\delta H(\ell)|\delta  H(\ell)|^{2}d\ell
=& \partial_{i}\B(H_{i}(  H_{j}^{2} )^{\varepsilon}-(H_{i} H_{j}^{2} )^{\varepsilon}\B)+2\partial_{i}(H_{i} H_{j}  )^{\varepsilon} H_{j}   -2H_{i}\partial_{i}  H_{j} ^{\varepsilon}H_{j},\ea\ee
and
\be\label{3.12}\ba
\int_{\mathbb{T}^{3}}\nabla\varphi_{\varepsilon}(\ell)\cdot\delta H(\ell)|\delta  v(\ell)|^{2}d\ell
 =& \partial_{i}\B(H_{i}( v_{j}^{2} )^{\varepsilon}-(H_{i} v_{j}^{2} )^{\varepsilon}\B)+2\partial_{i}(H_{i} v_{j}  )^{\varepsilon} v_{j}   -2H_{i}\partial_{i}  v_{j} ^{\varepsilon} v_{j}.
 \ea \ee
 With \eqref{3.9}-\eqref{3.12} in hand, repeating the previous argument  enables us to complete the proof.
\end{proof}
\section{Conclusion}
The Kolmogorov's law and  Yaglom's law are two of the few rigorous results in the theory of turbulence (see \cite{[AOAZ],[Biskamp],[Frisch],[AB],[Eyink2]}).
Here, we conclude the diverse versions of Yaglom's law for various turbulence models, such as the temperature equation, the inviscid MHD equations and the ideal incompressible
Euler equations, in the sense of Duchon-Robert \cite{[DR]} and Eyink \cite{[Eyink1]}. To this end, we deduce the corresponding dissipation terms, which resulted from the lack of smoothness of the solutions in  conservation relations of the energy, cross-helicity and helicity.

 Our conclusions, Theorem \ref{the1.1}-\ref{the1.4}, are in good agreement with the known results derived from K\'arm\'an-Howarth type equations. However, it seems that there is a little gap between Theorem \ref{the1.5}  and $4/3$ law \eqref{1.20} at present.
It is helpful to discover the new Yaglom's law in  turbulence. Indeed, based on this, seven new four-three laws
are present here. To the knowledge   of the authors, these laws  are not mentioned and existed in  previous works and it seems that Theorem \ref{the1.6}-\ref{the1.8}  are completely new. It should be pointed out that Theorem \ref{the1.7} is close to recent work \cite{[BT]} and Theorem \ref{the1.6} and \ref{the1.8} are  self-contained.

It is worth pointing out that the $8/15$ law was also obtained by Eyink in \cite{[Eyink1]}. Likewise, there exist $4/5$ law and  $8/15$ law in the MHD equations (see \cite{[PP1],[YRS],[PP2],[AB]}). Moreover, $2/15$ law for the helicity in the Euler equations has been deduced in \cite{[LPP],[KTM],[Chkhetiani],[Chkhetiani1],[Kurien]}. It is still an interesting question how to derive these laws in the sense of Duchon-Robert \cite{[DR]}.
\section*{Acknowledgements}

 Wang was partially supported by  the National Natural
 Science Foundation of China under grant (No. 11971446, No. 12071113   and  No.  11601492) and sponsored
by Natural Science Foundation of Henan.
 Wei was partially supported by the National Natural Science Foundation of China under grant (No. 11601423, No. 12271433).  Ye was partially supported by the National Natural Science Foundation of China  under grant (No.11701145) and sponsored
by Natural Science Foundation of Henan.

\end{document}